\documentclass[11pt]{article}

\usepackage{xypic}
\usepackage[top=1.5in, bottom=1.5in, left=1in, right=1in]{geometry}
\usepackage{amsgen}
\usepackage{amsmath,amsthm}
\usepackage{amstext}
\usepackage{amsbsy}
\usepackage{amsopn}
\usepackage{amsfonts}
\usepackage{amssymb}
\usepackage{eepic}
\usepackage{graphicx}
\usepackage{epsf}
\usepackage{pstricks}
\usepackage{hyperref}
\xyoption{all}

\def\tra#1{\smash{\mathop{\mid\kern
-1pt\joinrel\relbar\joinrel\relbar}\limits^{*}_{#1}}}
\def\longtra#1{\smash{\mathop{\mid\kern
-1pt\joinrel\relbar\joinrel\relbar\joinrel\relbar}\limits^{*}_{#1}}}
\def\vlongtra#1{\smash{\mathop{\mid\kern
-1pt\joinrel\relbar\joinrel\relbar\joinrel\relbar\joinrel\relbar}\limits^{*}_{#1}}}
\def\vvlongtra#1{\smash{\mathop{\mid\kern
-1pt\joinrel\relbar\joinrel\relbar\joinrel\relbar\joinrel\relbar\joinrel\relbar}\limits^{*}_{#1}}}
\def\vvvlongtra#1{\smash{\mathop{\mid\kern
-1pt\joinrel\relbar\joinrel\relbar\joinrel\relbar\joinrel\relbar\joinrel\relbar\joinrel\relbar}\limits^{*}_{#1}}}
\def\etra#1{\smash{\mathop{\mid\kern
-1pt\joinrel\relbar\joinrel\relbar}\limits_{#1}}}

\def\A{{\cal{A}}}

\def\iff{\Leftrightarrow}
\def\Rw{\Rightarrow}
\def\oo{\overline}
\def\wt{\widetilde}
\def\wh{\widehat}

\def\C{{\cal{C}}}

\def\L{{\cal{L}}}
\def\M{{\cal{M}}}
\def\N{\mathbb{N}}

\def\S{{\cal{S}}}

\def\aut{\mbox{Aut}}
\def\inn{\mbox{Inn}}
\def\stab{\mbox{Stab}}

\def\lam{\wh{\lambda}}

\def\fix{\mbox{Fix}}

\def\ker{\mbox{Ker}\,}

\def\fin{\mbox{Fin}}

\def\max{\mbox{max}}
\def\min{\mbox{min}}
\def\rank{\mbox{rank}}

\def\supp{\mbox{supp}}

\def\R{{\cal{R}}}
\def\G{{\cal{G}}}

\def\Z{\mathbb{Z}}

\def\p{\varphi}

\def\inv{^{-1}}

\def\bi{\begin{itemize}}
\def\ei{\end{itemize}}
\def\beq{\begin{equation}}
\def\eeq{\end{equation}}
\def\ba{\begin{array}}
\def\ea{\end{array}}

\newtheorem{theorem}{Theorem}[section]
\newtheorem{proposition}[theorem]{Proposition}
\newtheorem{lemma}[theorem]{Lemma}

\newtheorem{corollary}[theorem]{Corollary}

\newtheorem{questions}[theorem]{Questions}

\newtheorem{remark}[theorem]{Remark}

\newtheorem{example}[theorem]{Example}

\normalbaselines
\def\abstract#1{\par\bigskip
\begingroup\small
\baselineskip=12truept
\begin{center}ABSTRACT\end{center}
\par\medskip\par\noindent
\null\hfill\hbox{\vbox{\hsize=5truein\noindent#1}}
\hfill\null\par\endgroup\par}

\title{Extensions of automorphisms of self-similar groups}
\author{{\bf Francesco Matucci and Pedro V. Silva}}

\date{\today}

\begin{document}
\maketitle




\abstract{In this work we study automorphisms of synchronous self-similar groups, the existence of extensions to automorphisms of the full group of automorphisms of the infinite rooted tree on which these groups act on. When they do exist, we obtain conditions for the continuity of such extensions with respect to the depth metric, but we also construct examples of groups where such extensions do not exist. We study the case of the lamplighter group $\L_k = \mathbb{Z}_k \wr \mathbb{Z}$.
}

\section{Introduction}

Let $A$ be a finite alphabet with at least two symbols, and let $A^\omega$ be the corresponding Cantor space of all infinite strings.  A homeomorphism of~$A^\omega$ is said to be \emph{rational} if there exists a transducer (i.e.~Mealy machine) that implements the mapping on infinite strings. 
There are two different versions of the rational group, depending on the type of transducers we allow.  A transducer that outputs exactly one symbol for each symbol that it reads is said to be \emph{synchronous}, and the group of all homeomorphisms of $A^\omega$ determined by synchronous transducers is the \emph{synchronous rational group}~$\mathcal{S}_A$.  If we instead allow transducers to output any finite number of symbols at each step, we obtain the \emph{asynchronous rational group}~$\R_A$, which includes $\S_A$ as a subgroup.

Subgroups of the synchronous groups $\S_A$ have received much attention in the literature~\cite{Bartholdi-Grigorchuk-Nekrashevych-1,Bartholdi-Grigorchuk-Sunik-1,Nek}.  In addition to branch groups and self-similar groups such as the Grigorchuk group, several other well-known groups can be represented by synchronous automata, including free groups~\cite{Vorobets-Vorobets}, $\mathrm{GL}_n(\mathbb{Z})$ and its subgroups~\cite{Brunner-Sidki-1}, the solvable Baumslag-Solitar groups $BS(1, m)$~\cite{Bartholdi-Sunik-1}, and the lamplighter groups~$R \wr \mathbb{Z}$ with $R$ a finite ring~\cite{Silva-Steinberg-1}. In the case of the asynchronous rational
group, it has recently been proved that all hyperbolic groups embed into the group ~$\R_A$.

Automorphisms of self-similar groups have been studied in detail before
by Lavreniuk and Nekrashevych \cite{Lavreniuk-Nekrashevych-1} and  Bartholdi and Sidki \cite{Bartholdi-Sidki-1} giving sufficient conditions for automorphisms to be induced by conjugation in $\aut(T_A)$. In the present paper, we focus on the study of some self-similar groups and their automorphisms
and we investigate their continuity with respect to the \emph{depth metric} and whether they can be extended to automorphisms of the full group of tree automorphisms of the infinite binary rooted tree. More precisely, since $\mathcal{S}_A \le \aut(T_A)$,
we define the depth metric on $\aut(T_A)$ so that, for $\p,\psi \in \aut(T_A)$, we let
$$d(\p,\psi) = \left\{
\ba{ll}
2^{-\min\{ n \in \N \; {\big{\lvert}}\; \p|_{A^n} \neq \psi|_{A^n}
\}}&\mbox{ if }\p \neq \psi\\
0&\mbox{ if }\p = \psi
\ea
\right.$$
where $A^n$ denote the $n$-th level of the tree $T_A$. Given a self-similar group $G \le \aut(T_A)$ we will be concerned with
its closure and we will obtain some general results about $\oo{G}$ and about the continuity of extensions of automorphisms
of $G$, and show examples with continuous extensions. However, examples
with non-continuous extensions exist.

\begin{theorem}
There exists an example of a self-similar group $G\leq {\rm Aut}(T_A)$ and an automorphism $\theta:G \to G$ not admitting a continuous extension to $\oo{G}$.
\end{theorem}

We specialize our machinery to the particular case of the lamplighter group
$\L_k = \mathbb{Z}_k \wr \mathbb{Z}$ and get the following result.

\begin{theorem}\label{thm:all-lamps-are-uniform}
Every automorphism of $\L_k$ admits a continuous extension for the depth metric.
\end{theorem}

On route to prove Theorem \ref{thm:all-lamps-are-uniform} we need
to study the structure of the group $\aut(\L_k)$ and we recover another proof of
a theorem of Taback, Stein and Wong \cite{Stein-Taback-Wong-1} describing $\aut(\L_k)$ and 
another proof of necessary and sufficient conditions for $\aut(\L_k)$ to be finitely generated
which were first explored by Bachmuth, Baumslag, Dyer and Mochizuki \cite{BaBaDyMo}
in the more general setting of $\aut(G)$ for $G$ a metabelian group.
Our proofs are independent and elementary.



The paper is organized as follows: in Section \ref{sec:prelim} we recall all the relevant definitions, in Section \ref{sec:depth} 
we study the depth metric and the closure of self-similar groups in ${\rm Aut}(T_A)$ with respect to such metric, in Section
\ref{sec:continuous-extension} we study conditions for the extension of an automorphism to be continuous and produce an
example where such extension is indeed not continuous, in Section \ref{sec:lamplighter} we
focus on the lamplighter group $\L_k$ and determine completely its automorphism group determining that it is made of
continuous bijections. 

\section{Preliminaries \label{sec:prelim}}

Given a finite nonempty set $A$, let $A^*$ denote the free monoid on
$A$ and let $\varepsilon$ denote 
the empty word. For a word $u \in A^*$ we write $|u|$ to denote its \emph{length}.
Given $u,v \in A^*$, we say that $u$ is a {\em prefix}
of $v$ if $v = uw$ for some $w \in A^*$.

We may identify $A^*$ with the regular $A$-ary rooted
tree $T_A$, where $A^*$ is the set of nodes, $\varepsilon$ is the root
and $ua$ is a son of $u$ for all 
$u \in A^*$ and $a \in A$ (i.e. the father of $u \in A^+ = A^*
\setminus \{ \varepsilon \}$ is its maximal proper prefix). An {\em
  automorphism} of $T_A$ is then a 
permutation of $A^*$ which preserves length and the prefix relation. 
We shall often work with a canonical alphabet $A_n = \{
0,1,\ldots,n-1\}$, and we shall use the notation $T_n = T_{A_n}$.

Note that $\aut(T_A)$ is a subgroup of the symmetric group on
$A^*$. Every $\p \in \aut(T_A)$ induces an action on the {\em
  boundary} of $T_A$, i.e. the set $A^{\omega}$ of all right infinite
words $a_0a_1a_2\ldots$ Sometimes it is more convenient to consider
the action of $\aut(T_A)$ on $A^{\omega}$. 

The group $\aut(T_A)$ is canonically isomorphic to the infinite wreath product
$S_A \wr S_A \wr \ldots$, where $S_A$ denotes the symmetric group on
$A$. This infinite decomposition involves the so-called local
permutations, which we now define. 

Given $\p \in \aut(T_A)$ and $u \in A^*$, there exists some $\p_u \in
S_A$ such that
\beq
\label{local}
(ua)\p = (u\p)(a\p_u) \; \mbox{ for every }a \in A.
\eeq
We say that $\p_u$ is the \emph{local permutation} induced by
$\p$ at vertex $u$. We also define the {\em cone automorphism} $\p_{uA^*} \in \aut(T_A)$
through
$$(uv)\p = (u\p)(v\p_{uA^*})\quad (v \in A^*).$$
We say that $G \leq \aut(T_A)$ is a {\em self-similar} group if 
$$\forall \p \in G, \; \; \forall u \in A^*\hspace{.5cm} \p_{uA^*} \in G.$$
Finitely generated self-similar groups can be constructed with the
help of finite automata/transducers of a particular type.
Indeed, we define a {\em Mealy machine} to be a structure of the form $\M =
(A,Q,\delta,\lambda)$, where:
\bi
\item
$A$ is a finite nonempty set (alphabet);
\item 
$Q$ is a finite nonempty set (state set);
\item
$\delta:Q \times A \to Q$ is a function (transition function);
\item
$\lambda:Q \times A \to A$ is a function (output function).
\ei
We say that $\M$ is {\em invertible} if the mapping
$$\ba{rcl}
\lambda_q:A&\to&A\\
a&\mapsto&(q,a)\lambda
\ea$$
is a permutation for every $q \in Q$. We define an action $Q \times A^* \to Q$ recursively by
\bi
\item
$q\varepsilon = q$;
\item
$q(ua) = (qu,a)\delta \quad (u \in A^*,\; a \in A)$.
\ei
For every $q \in Q$, we extend $\lambda_q$ to a mapping $\lam_q:A^* \to
A^*$ by setting
\bi
\item
$\varepsilon\lam_q = \varepsilon$;
\item
$(ua)\lam_q = (u\lam_q)(a\lambda_{qu})
\quad (u \in A^*,\; a \in A)$. 
\ei
If $\M$ is invertible then $\lam_q$ is a permutation of $A^*$,
indeed $\lam_q \in \aut(T_A)$. The
{\em automaton group} generated by $\M$ is the (finitely generated)
subgroup of $\aut(T_A)$ generated by $\{ \lam_q \mid q \in Q\}$. It
will be denoted by $\G(\M)$. Automata groups are precisely the
finitely generated self-similar groups.

We present now some examples which will play an important role in the
paper.

\begin{example}
\label{adding}
The adding machine. 
\emph{It is the invertible Mealy machine $\A$ depicted by
$$
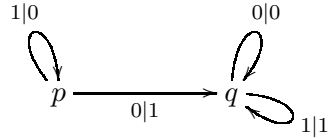
\xymatrix{
p \ar@(ul,u)^{1|0} \ar[rr]_{0|1} && q \ar@(u,ur)^{0|0} \ar@(r,dr)^{1|1} 
}$$
It is well known that $\G(\A) = \langle \lam_p \rangle$ is an infinite
cyclic group.}
\qed
\end{example}

\begin{example}
\label{cayley}
The Cayley machine of a finite group. 
\emph{Let $H$ be a finite group. The Cayley machine of $H$, introduced by
Krohn and Rhodes in \cite{Krohn-Rhodes-1}, is the invertible Mealy machine $\C_H
= (H,H,\delta,\lambda)$ defined by
$$(h,h')\delta = (h,h')\lambda = hh' \in H \quad (h,h' \in H).$$
If $H$ is abelian, Steinberg and the second author proved in \cite{Silva-Steinberg-1}
that $\G(\C_H) \cong H \wr \Z$, the wreath product of $H$ and $\Z$. If
$G = C_2$ is the group of order 2, we get the Cayley machine
$$
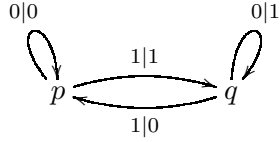
\xymatrix{
p \ar@(ul,u)^{0|0} \ar@/^/[rr]^{1|1} && q \ar@(u,ur)^{0|1} \ar@/^/[ll]^{1|0} 
}$$
and the {\em lamplighter group} $\L_2$.}
\qed
\end{example}

\section{The depth metric \label{sec:depth}}

Let $A$ be a finite nonempty alphabet. We define a metric on
$\aut(T_A)$ as follows. Given $\p,\psi \in \aut(T_A)$, let
$$d(\p,\psi) = \left\{
\ba{ll}
2^{-\min\{ n \in \N \; {\big{\lvert}}\; \p|_{A^n} \neq \psi|_{A^n}
\}}&\mbox{ if }\p \neq \psi\\
0&\mbox{ if }\p = \psi
\ea
\right.$$
where $A^n$ denote the $n$-th level of the tree $T_A$.
It is immediate that $d$ is indeed an ultrametric on $\aut(T_A)$,
which we call the {\em depth metric}. As it is remarked in \cite{Bartholdi-Grigorchuk-Sunik-1},
the following result follows from Tychonoff's Theorem.

\begin{proposition}
\label{compact}
The metric space $({\rm Aut}(T_A),d)$ is compact and therefore complete.
\qed
\end{proposition}

Let $G \leq \aut(T_A)$.
We denote by $\oo{G}$ the topological closure of $G$ in
$(\aut(T_A),d)$. Note that $\oo{G}$, being a closed subset of a
compact space, is itself compact (and therefore complete).

The restrictions of the depth metric to either $G$ or
$\oo{G}$ will still be denoted by $d$ and referred to as the depth metric.

It is easy to check that $\oo{G}$ consists precisely of those
$\p \in \aut(T_A)$ such that
$$\forall n \in \N \; \exists \p_n \in G: \p|_{A^{n}} =
\p_n|_{A^{n}},$$
and is indeed the completion of $(G,d)$. The study of this closure has been introduced
for branch groups by Bartholdi, Grigorchuk and \v{S}uni\'c (see
\cite[Definition 1.18]{Bartholdi-Grigorchuk-Sunik-1}) under the
name {\em tree completion}, which we also adopt.
Note that $\oo{G}$ is a subgroup of $\aut(T_A)$.

Given $n \in \N$, we define the $n$-{\em th level stabilizer} of
$G \leq \aut(T_A)$ to be
$$\stab_n(G) = \{ \p \in G \; {\big{\lvert}}\; \p|_{A^n} = id\}.$$
It is immediate that $\stab_n(G) \unlhd G$ for every $n$. Moreover,
we have a chain
$$G = \stab_0(G) \supseteq \stab_1(G) \supseteq \stab_2(G) \supseteq \ldots$$
where each subgroup has finite index and 
\beq
\label{inst}
\bigcap_{n\in \N} \stab_n(G) = \{ id\}.
\eeq
Therefore $G$ is residually finite.

Write $G_n = G/\stab_n(G)$ and consider the discrete topology
on the $G_n$. By considering the natural projections $G_m \to
G_n$ for all $m \geq n$, we obtain a projective system. It is
easy to see (see also \cite{Bartholdi-Grigorchuk-Sunik-1}) that $(\oo{G},d)$ is the
projective limit of the above projective system, hence it is a profinite
group (in particular, it is a topological group).

In general, the tree completion is a profinite group \cite{Bartholdi-Grigorchuk-Sunik-1, Grigorchuk-1}, but it needs not coincide with the profinite
completion of $G$, when we consider the profinite
metric. However, Grigorchuk remarks in \cite{Grigorchuk-1} that this happens if
every finite index subgroup of $G$ contains some $\stab_n(G)$
(that is, $G$ satisfies the congruence subgroup property). 

Note also that the topology on $G$ induced by $d$ is none other than the
{\em topology of pointwise convergence}: if we consider $A^*$ endowed
with the discrete topology, then $\lim_{n\to +\infty} \p_n = \p$ holds
in $G$ if and only if 
$$\forall u \in A^*\; \lim_{n\to +\infty} u\p_n = u\p,$$
i.e. each sequence $(u\p_n)_n$ is stationary with limit $u\p$.

\begin{lemma}
\label{ciss}
Let $G \leq {\rm Aut}(T_A)$ be self-similar. Then $\oo{G}$
is also a self-similar subgroup of ${\rm Aut}(T_A)$.
\end{lemma}

\begin{proof}
We have already noted that $\oo{G} \leq \aut(T_A)$. Let $\p \in \oo{G}$
and $u \in A^*$. There exists a sequence $(\p^{(n)})_n$ on $G$ such
that $\p = \lim_{n\to +\infty} \p^{(n)}$. Let $v \in A^*$. It follows easily from
continuity that
$$\ba{lll}
(u\p)(v\p_{uA^*})&=&(uv)\p = \lim_{n\to +\infty} (uv)\p^{(n)} =
\lim_{n\to +\infty} (u\p^{(n)})(v\p^{(n)}_{uA^*})\\
&=&(\lim_{n\to +\infty} u\p^{(n)})(\lim_{n\to +\infty}
v\p^{(n)}_{uA^*}) = (u\p)(\lim_{n\to +\infty}
v\p^{(n)}_{uA^*})),
\ea$$
hence $v\p_{uA^*} = \lim_{n\to +\infty} v\p^{(n)}_{uA^*}$. Since we
are dealing with the topology of pointwise convergence and $v$ is
arbitrary, we get 
$\p_{uA^*} = \lim_{n\to +\infty} \p^{(n)}_{uA^*}$. Since $\p^{(n)} \in G$ and
$G$ is self-similar, we have $\p^{(n)}_{uA^*} \in G$ and so $\p_{uA^*}
\in \oo{G}$. Therefore $\oo{G}$ is self-similar.
\end{proof}

Automata groups, being finitely generated, are always
countable. However, the next result shows that their tree completions
are countable only in trivial cases.

\begin{proposition}
\label{card}
Let $G \leq {\rm Aut}(T_A)$.
\bi
\item[(i)] If $G$ is finite, then $\oo{G} = G$.
\item[(ii)] If $G$ is infinite, then $\oo{G}$ is uncountable.
\ei
\end{proposition}

\begin{proof}
(i) True, since every finite subset of a metric space is closed.
(ii) True, since $\oo{G}$ is an infinite profinite group.

\end{proof}

We shall prove that $\oo{G} = \aut(T_A)$ for an automaton group $G$
only in trivial cases. But we 
start with the following lemma, where the {\em rank} of a finitely
generated group $G$ denotes the minimum cardinality of a generating
set of $G$.

We shall denote by $T_A^{(n)}$ the (rooted) subtree of $T_A$ induced
by the nodes $A^{\leq n}$. It is immediate to see that $\aut(T_A^{(n)}) \cong
(\aut(T_A))/\stab_n(\aut(T_A))$. The following result is well known (see for example
Brunner and Sidki \cite[Section 2]{Brunner-Sidki-1}).

\begin{lemma}
\label{rank}
Let $A$ be a finite alphabet with $|A| \geq 2$ and let $n \geq 1$. Then
${\rm rank}({\rm
  Aut}(T_A^{(n)})) = n$.
\end{lemma}

We can now prove the following.

\begin{proposition}
\label{proper}
Let $A$ be a finite alphabet with $|A| \geq 2$ and let $G \leq
\aut(T_A)$ be finitely generated. Then $\oo{G} < \aut(T_A)$.
\end{proposition}

\begin{proof}
Since $G$ is finitely generated, then
$\oo{G}$ is topologically finitely generated. However, it is well known that
$\aut(T_A)$ is not topologically finitely generated. 

We also offer a short
proof not involving profinite groups.
Suppose that $\oo{G} = \aut(T_A)$.
It follows easily from the
definition of closure that $G$ and $\oo{G}$ induce the same
automorphisms of $T_A^{(n)}$ for every $n \in \N$. Hence every $\psi \in
\aut(T_A^{(n)})$ is a restriction of some $\p \in G$. It follows that
$\aut(T_A^{(n)}) \cong G/\stab_n(G)$ for every $n \in \N$.

By Lemma \ref{rank}, $\rank(\aut(T_A^{(k+1)})) = k+1$ for every positive integer $k$. 
If we let $m = \rank(G)$, then we notice that
$$
m+1 = \rank(\aut(T_A^{(m+1)})) = \rank(G/\stab_{m+1}(G)) \leq \rank(G) = m,
$$
so we get a contradiction. Therefore $\oo{G} < \aut(T_A)$.
\end{proof}

In particular, we have $\oo{G} < \aut(T_A)$ for every automaton
group over an alphabet with at least two letters.

\section{Automorphisms of an automaton group \label{sec:continuous-extension}}

\subsection{General results}
It is a natural question to enquire which endomorphisms of an
automaton group admit a continuous extension to 
$(\oo{G},d)$. General topology yields the following result:

\begin{lemma}
\label{ucont}
Let $G \leq {\rm Aut}(T_A)$ and let $\theta:G \to G$ be a mapping.
Then the following conditions are equivalent:
\bi
\item[(i)] $\theta$ admits a continuous extension to $(\oo{G},d)$.
\item[(ii)] $\theta$ is uniformly continuous in $(G,d)$.
\ei
\end{lemma}

\begin{proof}
(i) $\Rw$ (ii). Since $(\aut(T_A),d)$ is compact by Proposition
\ref{compact} and $\oo{G}$ is a closed subset of $\aut(T_A)$, then
$(\oo{G},d)$ is itself compact (and therefore complete. It is the
completion of $(G,d)$). Hence a continuous extension of $\theta$
to $\oo{G}$ is thus a continuous map over a compact set, so it
is necessarily uniformly continuous with respect to $d$,
and its restriction is $\theta$.

(ii) $\Rw$ (i). Every uniformly continuous transformation of $(G,d)$
admits a continuous extension to its completion, which is precisely
$(\oo{G},d)$.
\end{proof}

Note that, if there is a continuous extension of the mapping $\theta:G \to G$ to its closure $\oo{G}$, 
it is unique. 
If an automorphism admits a continuous extension, we show that the
latter is also an automorphism.

\begin{proposition}
\label{ceia}
Let $G \leq {\rm Aut}(T_A)$ and let $\theta \in {\rm Aut}(G)$ be such that both $\theta$ and $\theta\inv$ are uniformly
continuous. Then its continuous extension
$\oo{\theta}:\oo{G} \to \oo{G}$ is also an automorphism.
\end{proposition}

\begin{proof}
Let $\p,\psi \in \oo{G}$. Then there exist sequences $(\p^{(n)})_n$
and $(\psi^{(n)})_n$ on $G$ such
that $\p = \lim_{n\to +\infty} \p^{(n)}$ and $\psi = \lim_{n\to +\infty} \psi^{(n)}$.
By continuity, we get
$$\ba{lll}
(\p\psi)\oo{\theta}&=&((\lim_{n\to +\infty} \p^{(n)})(\lim_{n\to
  +\infty} \psi^{(n)}))\oo{\theta} = (\lim_{n\to +\infty}
(\p^{(n)}\psi^{(n)}))\oo{\theta}\\
&=&\lim_{n\to +\infty} ((\p^{(n)}\psi^{(n)})\theta) = \lim_{n\to
  +\infty} ((\p^{(n)}\theta)(\psi^{(n)}\theta))\\ 
&=&(\lim_{n\to +\infty} (\p^{(n)}\theta))(\lim_{n\to
  +\infty} (\psi^{(n)}\theta)) = (\lim_{n\to +\infty} \p^{(n)})\oo{\theta}(\lim_{n\to
  +\infty} \psi^{(n)})\oo{\theta} = (\p\oo{\theta})(\psi\oo{\theta}),
\ea$$
thus $\oo{\theta}$ is a group homomorphism. 
Since $\oo{\theta\inv}$ is also a group homomorphism, it is easy to
see that $\oo{\theta}$ and $\oo{\theta\inv}$ are mutually
inverse. Therefore $\oo{\theta} \in \aut(\oo{G})$. 
\end{proof}

We can relate uniform continuity to stabilizers in the case of
automorphisms. 

\begin{proposition}
\label{ucsta}
Let $G \leq {\rm Aut}(T_A)$ and let $\theta:G \to G$ be an automorphism.
Then the following conditions are equivalent:
\bi
\item[(i)] $\theta$ is uniformly continuous in $(G,d)$;
\item[(ii)] $\forall m \in \N, \; \exists n \in \N: ({\rm
  Stab}_n(G))\theta \subseteq {\rm Stab}_m(G)$.
\ei
\end{proposition}

\begin{proof}
In fact, condition (i) is equivalent to the following equivalent statements
$$\ba{l}
\forall m \in \N, \; \exists n \in \N, \; \forall \p,\psi \in G\;
(d(\p,\psi) < 2^{-n} \Rw d(\p\theta,\psi\theta) < 2^{-m}) \Longleftrightarrow \\
{} \\
\forall m \in \N, \; \exists n \in \N, \; \forall \p,\psi \in G\;
(\p|_{A^n} = \psi|_{A^n} \Rw (\p\theta)|_{A^m} = (\psi\theta)|_{A^m}) \Longleftrightarrow \\
{} \\
\forall m \in \N, \; \exists n \in \N, \; \forall \p,\psi \in G\;
(\psi\inv\p \in \stab_n(G) \Rw (\psi\inv\p)\theta \in \stab_m(G)) \Longleftrightarrow \\
{} \\
\forall m \in \N, \; \exists n \in \N, \;
((\stab_n(G))\theta \subseteq \stab_m(G)).
\end{array}$$
\end{proof}

\begin{corollary}
\label{inner}
Let $G \leq {\rm Aut}(T_A)$ and let $\theta:G \to G$ be an inner automorphism.
Then $\theta$ is uniformly continuous in $(G,d)$.
\end{corollary}

\proof
By Proposition \ref{ucsta}, since $(\stab_n(G))\theta = \stab_n(G)$
for every inner automorphism $\theta$ of $G$.
\qed

\medskip
Recall that a subgroup $H$ of a group $G$ is \emph{characteristic} if $\varphi(H) \le H$,
for all $\varphi \in \mathrm{Aut}(G)$.

\begin{corollary}
\label{fulli}
Let $G \leq {\rm Aut}(T_A)$ be such that ${\rm Stab}_n(G)$ is characteristic for every $n \in \N$. Then every automorphism of $G$
is uniformly continuous in $(G,d)$.
\end{corollary}

\subsection{Aleshin's automaton and automorphisms \label{nfi}}

The following example shows that condition (ii) of Proposition \ref{ucsta} does not hold for
every automaton group. Let $G$ be the automaton group generated by Aleshin's automaton
$$
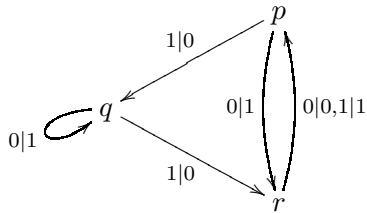
\xymatrix{
&&p \ar[dll]_{1|0} \ar@/_/[dd]_{0|1}\\
q \ar[drr]_{1|0} \ar@(l,dl)_{0|1}&&\\
&&r \ar@/_/[uu]_{0|0,1|1}
}$$
Then ${\rm Stab}_2(G)$ is not characteristic.
Indeed, Vorobets and Vorobets proved in \cite{Vorobets-Vorobets} that $G$ is the free
group on $\{ \lam_p, \lam_q, \lam_r \}$, hence 
$$\lam_p \mapsto \lam_r\lam_p, \quad \lam_q \mapsto \lam_q, \quad
\lam_r \mapsto \lam_r$$   
defines an elementary Nielsen automorphism $\theta$ of $G$. The action of the generators
of $G$ at depth 2 is described by the diagram
$$
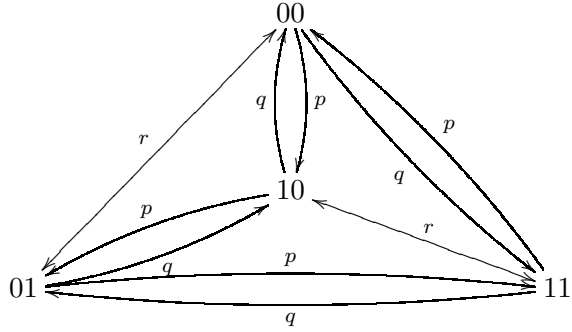
\xymatrix{
&&& 00 \ar@/_/[dddrrr]_q \ar@/^/[dd]^p \ar@{<->}[dddlll]_r &&&\\
&&&&&&\\
&&& 10 \ar@/^/[uu]^q \ar@/_/[dlll]_p &&&\\
01 \ar@/_/[urrr]_q \ar@/^/[rrrrrr]^p &&&&&& 11 \ar@/_/[uuulll]_p
\ar@/^/[llllll]^q \ar@{<->}[ulll]_r
}$$
It follows easily that $\lam_p\lam_q \in \stab_2(G)$, but
$(\lam_p\lam_q)\theta = \lam_r\lam_p\lam_q \notin
\stab_2(G)$. Therefore $\stab_2(G)$ is not characteristic.

\subsection{The adding machine and its automorphisms}
We consider now the adding machine to illustrate some of the
problems already introduced and some others we intend to
propose. Write $\p = \lam_p$ in the notation of Example \ref{adding}.
Since $\A$ generates the infinite cyclic group
$\G(\A) = \langle \p \rangle$, there exist only two automorphisms
of $\G(\A)$: the identity and the automorphism $\theta$ defined by
$$\p^n \mapsto \p^{-n}\quad(n \in \Z).$$
The adding machine takes its name from the fact that it reproduces
addition of integers in binary form in the following sense: if $u \in
A_2^m$, then $u\p^n \in A_2^m$ is the unique integer $v$ in binary
form such that $\wt{v}$ is congruent to $\wt{u}+n$ modulo $2^m$, where
$\wt{w}$ is the word $w$ read in reverse order. It follows that
$$\stab_m(\G(\A)) = \langle \p^{2^m} \rangle$$
for every $m \in \N$. But then $\stab_m(\G(\A))$ is a characteristic
subgroup of $\G(\A)$ for every $m \in \N$ and so every automorphism of $\G(\A)$
admits a continuous extension to $\oo{G}$ by Lemma \ref{ucont} and Corollary
\ref{fulli}.

We also claim that every automorphism of $\G(\A)$ can be actually
extended to an automorphism of $\aut(T_2)$. This is trivial for the
identity automorphism, so we only have to consider $\theta$. We show
that $\theta$ is the restriction to $\G(\A)$ of the {\em mirror image}
automorphism $\mu$ of $\aut(T_2)$.

Each $\p \in \aut(T_2)$ is fully determined by the local permutations
$\p_u$ $(u \in A_2^*)$. Let $\sigma \in \aut(A_2^*)$ interexchange $0$
and $1$. We define $\p\mu \in \aut(T_2)$ by
\beq
\label{rank1}
(\p\mu)_u = \p_{u\sigma} \quad (u \in A_2^*).
\eeq
We must prove that $\mu$ is an automorphism of $\aut(T_2)$. We start by
showing that
\beq
\label{mi1}
u(\p\mu)\sigma = u\sigma\p
\eeq
for all $\p \in \aut(T_2)$ and $u \in A_2^*$. We use induction on
$|u|$. The claim holds trivially for $u = \varepsilon$, hence we
assume that $u = va$ with $v \in A_2^*$ and $a \in A_2$, and that
(\ref{mi1}) holds for $v$. We have
$$u(\p\mu)\sigma = (va)(\p\mu)\sigma =
((v(\p\mu))(a(\p\mu)_v))\sigma =
(v(\p\mu)\sigma)(a\p_{v\sigma}\sigma).$$
By the induction hypothesis we get
$$u(\p\mu)\sigma = (v\sigma\p)(a\sigma\p_{v\sigma}) =
((v\sigma)(a\sigma))\p = (va)\sigma\p = u\sigma\p,$$
where $\p_{v\sigma} \sigma =\sigma\p_{v\sigma}$ since $S_{A_2}$ is abelian.
Therefore (\ref{mi1}) holds.

Now let $\p,\psi \in \aut(T_2)$. For every $u \in A_2^*$,
(\ref{rank1}) and (\ref{mi1}) yield
$$((\p\psi)\mu)_u = (\p\psi)_{u\sigma} = \p_{u\sigma}\psi_{u\sigma\p} =
\p_{u\sigma}\psi_{u(\p\mu)\sigma} = 
(\p\mu)_{u}(\psi\mu)_{u(\p\mu)} = ((\p\mu)(\psi\mu))_u,$$
hence $(\p\psi)\mu = (\p\mu)(\psi\mu)$ and so $\mu$ is a group
homomorphism. Since $\mu$ is clearly bijective, it is an automorphism
of $\aut(T_2)$.

Finally, we show that $\theta = \mu|_{\G(\A)}$. Writing $\p = \lam_p$,
it suffices to show that $\p\theta = \p\mu$. Let $u \in A_2^*$. We
must show that
\beq
\label{mi2}
(\p\theta)_u = (\p\mu)_u.
\eeq
Note that (\ref{rank1}) and the fact that the local functions are in $S_{A_2}$ imply
$$
(\p\inv)_u = \p_{u\p\inv},
$$
hence 
\beq
\label{mi3}
(\p\theta)_u = (\p\inv)_u = \p_{u\p\inv}.
\eeq
On the other hand,
\beq
\label{mi4}
(\p\mu)_u = \p_{u\sigma}.
\eeq
It follows easily from the construction of the adding machine that 
$$\p_v = \left\{
\ba{ll}
id&\mbox{ if }v \in 0A_2^*\\
(01)&\mbox{ otherwise}
\ea
\right.$$
Now $u\sigma \in 0A_2^*$ if and only if $u \in 1A_2^*$ if and only if
$u\p\inv \in 0A_2^*$, hence $\p_{u\sigma} = \p_{u\p\inv}$ and so (\ref{mi2})
follows from (\ref{mi3}) and (\ref{mi4}). Therefore 
$\theta = \mu|_{\G(\A)}$ and so every automorphism of $\G(\A)$ is a
restriction of some automorphism of $\aut(T_2)$.

We can also compute the subgroup of fixed points
$$\fix(\theta) = \{ \psi \in \G(\A) \mid \psi\theta = \psi \}.$$ 
Since $\p^n\theta = \p^{-n}$, we have $\fix(\theta) = \{ id \}$, hence
finitely generated. We shall see in the next section that this is not
always the case.

\subsection{A non-uniformly continuous automorphism}

In this section we describe a construction to show that direct product of automata groups is itself 
an automaton group. The result is well known (for example, see the construction mentioned immediately after Theorem 2.2 of the survey \cite{Bartholdi-Silva-1}
and Remark \ref{thm:direct-product-known} below). However, to the best of our knowledge
the construction described in the proof below is new and potentially of independent interest
and we will use it to show an application afterwards.

\begin{lemma}
\label{thm:direct-product-automata}
If $G$ and $H$ are automata groups, then 
$G \times H$ is an automaton group too.
\end{lemma}

\begin{proof}
We can assume that $G=\mathcal{G}(\mathcal{A}) \le \aut(T_A)$ and $H=\mathcal{G}(\mathcal{B}) \le \aut(T_B)$ 
are automata groups where the automata $\mathcal{A}$ and $\mathcal{B}$ are such that the alphabets $A$ and $B$ are disjoint.
We now  construct two new automata $\mathcal{A}'$ and $\mathcal{B}'$ on the same alphabet $C=A\cup B$.
The automaton $\mathcal{A}'$ has the same states of $\mathcal{A}$ and has the same transition and output function on the elements of $A$, while $(b,q)\delta_{\mathcal{A}'}=q$ and $(b,q)\lambda_{\mathcal{A}'}=b$ for all states $q \in \mathcal{A}'$ and $b\in B$.
In a similar fashion, $\mathcal{B}'$ has the same states of $\mathcal{B}$ and has the same transition and output function on the elements of $B$, while $(a,q)\delta_{\mathcal{B}'}=q$ and $(a,q)\lambda_{\mathcal{B}'}=a$ for all states $q \in \mathcal{B}'$ and $a\in A$. Hence the transition
and output functions behave as shown in the picture below:

$$
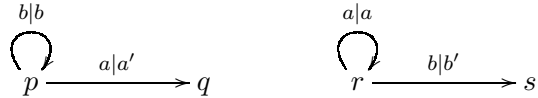
\xymatrix{
p \ar@(ul,ur)^{b|b} \ar[rr]^{a|a'} && q  
}
\qquad \qquad
\xymatrix{
r \ar@(ul,ur)^{a|a} \ar[rr]^{b|b'} && s 
}
$$

Let $\mathcal{C}$ be the automaton given by the disjoint union of $\mathcal{A}'$ and $\mathcal{B}'$.
Observe that $\mathcal{G}(\mathcal{A}')=\langle Q_{\mathcal{A}} \rangle$
and $\mathcal{G}(\mathcal{B}')=\langle Q_{\mathcal{B}} \rangle$, where $Q_{\mathcal{A}}$ is the set of rational maps on the alphabet $A$
given by taking each state of $\mathcal{A}$ as the initial state of the automaton. We also define $Q_{\mathcal{B}}$ is accordingly.
By construction of $\mathcal{A}'$ and $\mathcal{B}'$, it is clear that $gh=hg$ for all $g \in G$ and $h \in H$ and that
$G \cap H = \{1\}$, therefore we have that $\mathcal{G}(\mathcal{C}) \cong G \times H$ and so $G \times H$ is an automaton group.
\end{proof}

\begin{remark}
\label{thm:direct-product-known}
The direct product construction
mentioned in \cite{Bartholdi-Silva-1} constructs the direct product of automata groups as an automaton group by taking the direct product of the alphabets and
produces a connected automaton for $G \times H$, while the construction above keeps previously existing connected components disjoint from each other.
\end{remark}

We will now use Proposition \ref{ucsta} and Lemma \ref{thm:direct-product-automata} to construct an example of a non-uniformly continuous
automorphism of an automaton group. Consider the following automaton $\mathcal{A}$

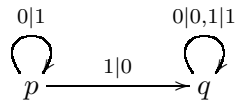
$$
\xymatrix{
p \ar@(ul,ur)^{0|1} \ar[rr]^{1|0} && q \ar@(ul,ur)^{0|0, 1|1}
}
$$
and let $p$ be the rational function arising from the automaton with initial state $p$. By construction
\[
1^\omega \to^p 01^\omega \to^p 101^\omega \to^p 001^\omega \to^p 1101^\omega \to^p \ldots
\]
and so, since each image through $p$ will have a tail given by $01^\omega$, then $p$ has infinite order and $\mathcal{G}(\mathcal{A}) = \langle p \rangle \cong \mathbb{Z}$. Consider the following automaton $\mathcal{B}$

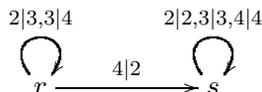
$$
\xymatrix{
r \ar@(ul,ur)^{2|3, 3|4} \ar[rr]^{4|2} && s \ar@(ul,ur)^{2|2, 3|3, 4|4}
}
$$
and let $r$ be the rational function coming from the automaton with initial state $r$. By construction
\[
4^\omega \to^r 24^\omega \to^r 324^\omega \to^r 4324^\omega \to^r 2324^\omega \to^r 
34324^\omega \to^r \ldots
\]
and so, since each image through $p$ will have a tail given by $24^\omega$,
then $r$ has infinite order and $\mathcal{G}(\mathcal{B}) = \langle r \rangle \cong \mathbb{Z}$.
If we build the automaton $\mathcal{C}$ as in the proof of Lemma \ref{thm:direct-product-automata},
then $\mathcal{G}(\mathcal{C}) = \langle p,r\rangle \cong \mathbb{Z}^2$.
Let $\theta \in \mathrm{Aut}(\mathcal{G}(\mathcal{C}))$ be the automorphism swapping $p$
with $r$. We argue by contradiction that $\theta$ is uniformly continuous.
By Proposition \ref{ucsta} there exists a positive integer $n$ such that $(\mathrm{Stab}_n(\mathcal{G}(\mathcal{C})))\theta \subseteq \mathrm{Stab}_1(\mathcal{G}(\mathcal{C}))$.

Observe that for every $g \in \mathrm{Aut}(T_2)$ we have $g^{2^n} = id$ at level $n$. In fact, by induction, if $g^{2^{n-1}}$ fixes level $n-1$ pointwise, then even if it swapped some of the children of  level $n-1$, we have that $(g^{2^{n-1}})^2=g^{2^n}$ would fix all such children. Hence, in particular, $p^{2^n} \in 
\mathrm{Stab}_n(\mathcal{G}(\mathcal{C}))$ and thereforre
$r^{2^m}=(p^{2^m})\theta \in \mathrm{Stab}_1(\mathcal{G}(\mathcal{C}))$.
However, since $r$ acts as the $3$-cycle $(2 \; 3 \; 4)$ at level $1$, then 
$(2 \; 3 \; 4)^{2^m} =id$ and so $3=|(2 \; 3 \; 4)|$ divides $2^m$, which is a contradiction and
implying that $\theta$ is not uniformly continuous.

\section{Cayley machines of finite abelian groups \label{sec:lamplighter}}

We follow the notation from \cite{Silva-Steinberg-1}. Let $(H,+)$ be a finite nontrivial abelian group in
additive notation. We consider the action of $\G(\C_H)$ on the
boundary of the tree $T_H$. For
all $h,x_0,x_1,\ldots \in H$, we have 
$$(x_0,x_1,x_2,\ldots)\lam_h = (h+x_0,h+x_0+x_1,h+x_0+x_1+x_2,\ldots)$$
Let $\xi = \lam_0$. For every $h \in H$, let $\alpha_h =
\lam_h\xi\inv$. Since $\G(\C_H) = \langle \lam_h \mid h \in H\rangle$,
we also have
$$\G(\C_H) = \langle \xi, \alpha_h \mid h \in H\rangle.$$
Note that
$$\ba{rcl}
(x_0,x_1,x_2,\ldots)\xi&=&(x_0,x_0+x_1,x_0+x_1+x_2,\ldots),\\
(x_0,x_1,x_2,\ldots)\xi\inv&=&(x_0,-x_0+x_1,-x_1+x_2,\ldots),\\
(x_0,x_1,x_2,\ldots)\alpha_h&=&(h+x_0,x_1,x_2,\ldots).
\ea$$
We can identify $(x_0,x_1,x_2,\ldots) \in H^{\omega}$ with the formal
power series $X = \sum_{n = 0}^{\infty} x_nt^n \in H[[t]]$. Then we have
\beq
\label{powse}
X\xi^n = X(1-t)^{-n}, \quad X\xi^n\alpha_h\xi^{-n} = X + h(1-t)^n
\eeq
for all $X \in H[[t]]$, $n \in \Z$ and where of course we mean $(1-t)^{-1}=\sum_{i=0}^{\infty}t^i$. By \cite{Silva-Steinberg-1}, every $\p \in
\G(\C_H)$ admits a unique factorization of the form
$$\p =
(\xi^{n_1}\alpha_{h_1}\xi^{-n_1})\ldots(\xi^{n_k}\alpha_{h_k}\xi^{-n_k})\xi^m,$$
with $k \in \N$; $n_i,m \in \Z$; $n_1 < \ldots < n_k$; $h_i \in
H\setminus \{ 1 \}$. By a slight abuse of notation, we identify elements $h \in H$ with the
maps $\alpha_h$ and thus we observe that
$$\langle H,\xi \mid [\xi^mh\xi^{-m}, \xi^nh'\xi^{-n}]\; (h,h' \in H;\;
m,n \in \Z)\rangle$$
constitutes a relative presentation of $\G(\C_H) \cong H \wr \Z$.

We now fix  $k \geq 2$ and we consider the 
lamplighter group $\L_k = \G(\C_{C_k})$. We identify $C_k$ with $\Z_k$, the integers modulo $k$ under addition. We denote by $\Z_k^*$ the subgroup of units of $(\Z^k, \cdot)$ (i.e. the integers modulo $k$ coprime with $k$). Writing $\alpha = \alpha_1$, we get a presentation of $\L_k$ of the form
\beq
\label{lapre}
\langle \alpha,\xi \mid \alpha^k, [\xi^m\alpha\xi^{-m},
  \xi^n\alpha\xi^{-n}]\; (m,n \in \Z)\rangle.
\eeq
Note that $\xi^{m_0}\alpha^{n_1}\xi^{m_1}\ldots \alpha^{n_k}\xi^{m_k}$
has finite order if and only if $m_0 + m_1 + \ldots + m_k = 0$. Let
$\fin(\L_k)$ denote the subset of all elements of $\L_k$ of finite
order. Then $\fin(\L_k)$ is an abelian normal subgroup of $\L_k$ and $\L_k/\fin(\L_k)$ 
is infinite cyclic. Moreover, $\fin(\L_k)$ is a direct sum of countably
many cyclic groups of order $k$.

For every $m \in \Z$, we write $\beta_m = \xi^m\alpha\xi^{-m}$. It follows easily from (\ref{lapre}) that each $x \in \L_k$ can be written uniquely in the form
$$x = \left(\prod_{n \in \Z} \beta^{i_n}_n\right)\xi^r,$$
where $r \in \Z$ and $(i_n)_n$ is a mapping $\Z \to \Z_k$ with finite support (i.e. only finitely many terms $i_n$ are nonzero). Moreover, $x \in \fin(\L_k)$ if and only if $r = 0$. 

How can we characterize the endomorphisms $\p$ of $\L_k$? Since $\L_k$ is generated by $\alpha$ and $\xi$, then $\p$ is fully determined by the images of $\alpha$ and $\xi$. It follows easily from (\ref{lapre}) that $\alpha\p$ can be any element $y \in \fin(\L_k)$ and $\xi\p$ can be any element $z \in \L_k$: indeed, the equalities $y^k = 1$ and $[z^myz^{-m}, z^nyz^{-n}] = 1$ follow from the fact that $z^nyz^{-n} \in \fin(\L_k)$ for every $n$, and $\fin(\L_k)$ is an abelian group of exponent $k$.

Given mappings $(i_n)_n$ and $(j_n)_n$ from $\Z$ to $\Z_k$ with finite support and
$r \in \Z$, and in view of the normal form defined before, the correspondence
$$\alpha \mapsto \prod_{n \in \Z} \beta_n^{i_n},\quad \xi \mapsto \left(\prod_{n \in \Z} \beta_n^{j_n}\right)\xi^{r}$$
induces an endomorphism of $\L_k$, which we denote by $\p_{(i_n)_n,(j_n)_n,r}$. Conversely, every endomorphism can be written uniquely in this form. We next show that
\beq
\label{autlg3}
\beta_m\p_{(i_n)_n,(j_n)_n,r} = \prod_{n \in \Z} \beta_{n}^{i_{n-rm}}
\eeq
holds for every $m \in \Z$. This holds trivially for $m = 0$. Assume
now that $m > 0$ and (\ref{autlg3}) holds for $m-1$. Since $\fin(\L_k)$ is abelian, we get
$$\begin{array}{lll}
\beta_m\p_{(i_n)_n,(j_n)_n,r}&=&(\prod_{n \in \Z} \beta_n^{j_n})\xi^{r}(\prod_{n \in \Z} 
\beta_{n}^{i_{n-r(m-1)}})\xi^{-r}(\prod_{n \in \Z} \beta_n^{-j_n})
= \xi^{r}(\prod_{n \in \Z} \beta_{n}^{i_{n-r(m-1)}})\xi^{-r}\\
&=&\prod_{n \in \Z} \beta_{n+r}^{i_{n-r(m-1)}} = \prod_{n \in \Z} \beta_{n}^{i_{n-rm}}.
\end{array}$$
Thus (\ref{autlg3}) holds for every $m > 0$. Similarly, we show
that it holds for $m < 0$. 

Now we can derive 
\beq
\label{autlg4}
\beta_m\p_{(i_n)_n,(j_n)_n,r}\p_{(i'_n)_n,(j'_n)_n,r'}  = \prod_{n \in \Z} \beta_{n}^{\sum_{k \in \Z} i'_{n-r'k}i_{k-rm}}
\eeq
Indeed, using (\ref{autlg3}) twice we get
$$\begin{array}{lll}
\beta_m\p_{(i_n)_n,(j_n)_n,r}\p_{(i'_n)_n,(j'_n)_n,r'}&=&\prod_{k \in \Z} \beta_{k}^{i_{k-rm}}\p_{(i'_n)_n,(j'_n)_n,r'}  =
\prod_{k \in \Z} ( \prod_{n \in \Z} \beta_{n}^{i'_{n-r'k}})^{i_{k-rm}}\\
&=&\prod_{n \in \Z} \beta_{n}^{\sum_{k \in \Z} i'_{n-r'k}i_{k-rm}}.
\end{array}$$

It follows easily from (\ref{autlg4}) that 

\beq
\label{autlg8}
\p_{(i_n)_n,(0)_n,1}\p_{(i'_n)_n,(0)_n,1} = \p_{(i'_n)_n,(0)_n,1}\p_{(i_n)_n,(0)_n,1}.
\eeq
Indeed, $\xi$ is fixed by both endomorphisms, and
$$\alpha\p_{(i_n)_n,(0)_n,1}\p_{(i'_n)_n,(0)_n,1} = \prod_{n \in \Z} \beta_{n}^{\sum_{k \in \Z} i'_{n-k}i_{k}} = \prod_{n \in \Z} \beta_{n}^{\sum_{k \in \Z} i'_{k}i_{n-k}} = \alpha\p_{(i'_n)_n,(0)_n,1}\p_{(i_n)_n,(0)_n,1}.$$

We also get
\beq
\label{autlg9}
\p_{(i_n)_n,(0)_n,1}\p_{(i'_n)_n,(0)_n,1} = \p_{(\sum_{k \in \Z} i'_{n-k}i_{k})_n,(0)_n,1}.
\eeq

This formula can be generalized for an arbitrary number of factors. For every $q \geq 2$, we have
\beq
\label{autlg10}
\p_{(i_n^{(1)})_n,(0)_n,1}\ldots \p_{(i_n^{(q)})_n,(0)_n,1} = \p_{(\sum_{k_1 + \ldots + k_q = n} i_{k_1}^{(1)} \ldots i_{k_q}^{(q)})_n,(0)_n,1}.
\eeq
where the $k_j$ take values in $\Z$. Indeed, the case $q = 2$ is just (\ref{autlg9}) rewritten, so we assume that (\ref{autlg10}) holds for $q \geq 2$ and we prove it for $q+1$. Using (\ref{autlg9}) and the induction hypothesis, we get
$$\begin{array}{lll}
p_{(i_n^{(1)})_n,(0)_n,1}\ldots \p_{(i_n^{(q+1)})_n,(0)_n,1}&=&\p_{(\sum_{k_1 + \ldots + k_q = n} i_{k_1}^{(1)} \ldots i_{k_q}^{(q)})_n,(0)_n,1}\p_{(i_n^{(q+1)})_n,(0)_n,1}\\
&=&\p_{(\sum_{k_{q+1} \in \Z}(\sum_{k_1 + \ldots + k_q = n-k_{q+1}} i_{k_1}^{(1)} \ldots i_{k_q}^{(q)})i_{k_{q+1}}^{(q+1)})_n,(0)_n,1}\\
&=&\p_{(\sum_{k_1 + \ldots + k_{q+1} = n} i_{k_1}^{(1)} \ldots i_{k_{q+1}}^{(q+1)})_n,(0)_n,1}
\end{array}$$
and so (\ref{autlg10}) holds.

It is much harder to identify the automorphisms of $\L_k$. We start from studying the following subgroup of $\aut(\L_k)$
$$\stab_{\L_k}(\xi) = \{ \p \in \aut(\L_k) \mid \xi\p = \xi \}.$$

It is straightforward to check that (\ref{autlg9}) is equivalent to the following, for given $\p,\p' \in \stab_{\L_k}(\xi)$:
\beq
\label{autlg11}
\mbox{if $\alpha\p = \prod_{j=1}^r \beta_{m_j}^{i_j}$ and $\alpha\p' = \prod_{\ell =1}^s \beta_{n_{\ell}}^{i'_{\ell}}$, then
$\alpha\p\p' = \prod_{j=1}^r \prod_{\ell =1}^s \beta_{m_j+n_{\ell}}^{i_ji'_{\ell}}$}.
\eeq
We introduce some notation for specific automorphisms in $\stab_{\L_k}(\xi)$. Let $(\varepsilon_n)_n$ be the mapping from $\Z$ to $\Z_k$ defined by 
$$\varepsilon_n = \left\{
\begin{array}{ll}
1&\mbox{ if }n = 0\\
0&\mbox{ otherwise}
\end{array}
\right.$$
This mapping will be handy throughout the rest of the section.
Let $\lambda$ denote the inner automorphism of $\L_k$ defined by $x\lambda = \xi x\xi\inv$. Clearly, $\lambda \in \stab_{\L_k}(\xi)$.

For every $j \in \Z^*_k$, we define $\eta_j \in \stab_{\L_k}(\xi)$ by $\alpha\eta_j = \alpha^j$. Note that $\alpha^j$ is well defined since $\alpha^k = 1$. If $j\inv$ denotes the inverse of $j$ in $\Z^*_k$, then $\eta_j\eta_{j\inv} = 1 = \eta_{j\inv}\eta_j$ and so $\eta_j$ is indeed an automorphism.

Finally, let $m \in \Z$ and $j \in \Z^*_k$. We define $\gamma_{m,j} \in \stab_{\L_k}(\xi)$ by $\alpha\gamma_{m,j} = \beta_m^j$. 
Note that 
\beq
\label{autlg7}
\lambda = \gamma_{1,1} \quad \mbox{and} \quad \eta_j = \gamma_{0,j}.
\eeq
Since $\alpha\lambda^m = \beta_m$ holds for every $m \in \Z$, it follows easily that
\beq
\label{autlg5}
\gamma_{m,j} = \lambda^m\eta_j.
\eeq

\begin{lemma}
\label{stap}
Let $p$ be a positive prime. Then 
$${\rm Stab}_{\L_p}(\xi) = \{ \gamma_{m,j} \mid m \in \Z,\, j \in \Z^*_p \} = \langle \lambda, \eta_2, \ldots, \eta_{p-1} \rangle.$$
\end{lemma}

\begin{proof}
We had already established that $\gamma_{m,j} \in \stab_{\L_p}(\xi)$ for all $m,j$. Conversely, let $\psi \in \stab_{\L_p}(\xi)$. Then we may write
$\psi = \p_{(i_n)_n,(0)_n,1}$ for some mapping $(i_n)_n$ from $\Z$ to $\Z_k$ with finite support. Since $\xi\psi\inv = \xi$, we may write $\psi\inv = \p_{(i'_n)_n,(0)_n,1}$ for some $(i'_n)_n$. Since $\alpha\psi\psi\inv = \alpha$, it follows from (\ref{autlg4}) that
\beq
\label{autlg6}
\sum_{k \in \Z} i'_{n-k}i_{k} = \left\{
\begin{array}{ll}
1&\mbox { if }n = 0\\
0&\mbox{ otherwise }
\end{array}.
\right.
\eeq
Assume that
$$\supp((i_n)_n) = \{ n \in \Z \mid i_n \neq 0 \} = \{ s_1,\ldots, s_m\}$$
with $s_1 < \ldots < s_m$. Assume also that $\supp((i'_n)_n) = \{ s'_1,\ldots, s'_{m'}\}$
with $s'_1 < \ldots < s'_{m'}$. Then $k > s_1$ if and only if $s_1 + s'_1 - k < s'_1$, hence
$$\sum_{k \in \Z} i'_{s_1+s'_{1}-k}i_{k} = i'_{s'_1}i_{s_1} \neq 0.$$
Similarly, $k < s_m$ if and only if $s_m + s'_{m'} - k > s'_{m'}$, hence
$$\sum_{k \in \Z} i'_{s_m+s'_{m'}-k}i_{k} = i'_{s'_{m'}}i_{s_m} \neq 0.$$
In view of (\ref{autlg6}), we get $s_1 + s'_1 = s_m + s'_{m'}$, hence $m = m' = 1$ and so $\alpha\psi = \beta_{s_1}^{i_{s_1}}$. Thus $\psi = \gamma_{s_1,i_{s_1}}$ and the first equality of the lemma is established. The second follows from (\ref{autlg7}) and (\ref{autlg5}). 
\end{proof}

We consider now the case of powers of primes. Let $p$ be a positive prime and $s \geq 2$. Given $m \in \Z \setminus \{ 0\}$ and $r \in \Z_{p^s}$, we define an endomorphism $\delta_{m,pr}$ of $\L_{p^s}$ by $\alpha\delta_{m,pr} = \alpha\beta_m^{pr}$.

\begin{lemma}
\label{staps}
Let $p$ be a positive prime and $s \geq 2$. Then 
$$\begin{array}{lll}
{\rm Stab}_{\L_{p^s}}(\xi)&=&\{ \p_{(i_n)_n,(0)_n,1} \mid \mbox{ there exists a unique $m \in \Z$ such that $p$ does not divide }i_m\}\\
&=&\langle \lambda, \eta_j, \delta_{m,pr} \mid j \in \Z^*_{p^s},\, m \in \Z \setminus \{ 0\},\, r \in \Z_{p^s} \rangle.
\end{array}$$
\end{lemma}

\begin{proof}
Let $\psi = \p_{(i_n)_n,(0)_n,1} \in \stab_{\L_{p^s}}(\xi)$. 
Write $\psi\inv = \p_{(i'_n)_n,(0)_n,1}$. In view of (\ref{autlg4}), (\ref{autlg6}) holds,
therefore $\sum_{k \in \Z} i'_{-k}i_k = 1 \, ({\rm mod} \, p^s)$ and thus 
$\sum_{k \in \Z} i'_{-k}i_k = 1  \, ({\rm mod} \, p)$. It follows that there exists some $m \in \Z$ such that $p$ does not divide $i'_{-m}i_m$. 
We may assume that $r$ and $m$ are respectively the leftmost and the rightmost integers 
$n \in \mathbb{Z}$ such that $p$ does not divide $i_n$,
so that $r \leq m$. Let $r'$ and $m'$ be respectively the leftmost and the rightmost integers $n \in \Z$ such that $p$ does not divide $i'_n$, so that $r' \leq m'$.

We claim that $\sum_{k \in \Z} i'_{m'+m-k}i_{k} \neq 0 \, ({\rm mod} \, p)$. Indeed, if $k > m$, then $p$ divides $i_k$, and $p$ divides $i'_{m'+m-k}$ if $k < m$ (equivalent to $m'+m-k > m'$). However, $p$ does not divide $i'_{m'+m-m}i_{m}$ since both factors are invertible in $\Z_{p^s}$. Thus $\sum_{k \in \Z} i'_{m'+m-k}i_{k} \neq 0 \, ({\rm mod} \, p)$. Similarly, $\sum_{k \in \Z} i'_{r'+r-k}i_{k} \neq 0 \, ({\rm mod} \, p)$. In view of (\ref{autlg6}), we get $m' + m = r' + r$. Hence $m'+m = r'+r \leq m' + r$, yielding $m \leq r$ and consequently $m = r$. Therefore there exists a unique $m \in \Z$ such that $p$ does not divide $i_m$. 

For the second part of the proof, let us now now fix an $m \in \Z \setminus \{ 0\}$ and an $r \in \Z_{p^s}$, to show that $\delta_{m,pr} \in \stab_{\L_{p^s}}(\xi)$. In order to do so, we introduce
$$\Phi_m = \{ \p_{(i_n)_n,(0)_n,1} \mid i_0 = 1\mbox{ and $i_n = 0$ if }\frac{n}{m} \notin \N \}.$$
We claim that $\Phi_m$ is a submonoid of endomorphisms of $\L_{p^s}$. It certainly contains the identity. Let $i''_n = \sum_{k \in \Z} i'_{n-k}i_{k}$. By
(\ref{autlg9}), we have $\p_{(i_n)_n,(0)_n,1}\p_{(i'_n)_n,(0)_n,1} = \p_{(i''_n)_n,(0)_n,1}$. Now $i''_0 = \sum_{k \in \Z} i'_{-k}i_{k}$. If $k \neq 0$, then $\frac{-k}{m} \notin \N$ or $\frac{k}{m} \notin \N$, hence $i'_{-k}i_{k} = 0$. It follows that $i''_0 = i'_0i_0 = 1$. Suppose now that $\frac{n}{m} \notin \N$. For every $k \in \Z$, we have $\frac{n-k}{m} \notin \N$ or $\frac{k}{m} \notin \N$, hence $i'_{n-k}i_{k} = 0$ and so $i''_n = 0$. Thus $\Phi_m$ is a monoid.

We define now a mapping $\Lambda$ from $\Phi_m$ to the polynomial ring $\Z_{p^s}[x]$ as follows. Given $\p_{(i_n)_n,(0)_n,1} \in \Phi_m$, let
$$\p_{(i_n)_n,(0)_n,1}\Lambda = \sum_{k \in \N} i_{mk} x^k.$$
Since $i_n$ can only be nonzero for $n \in m\N$, the mapping $\Lambda$ is injective. We show that it is a monoid endomorphism with respect to the multiplicative structure of $\Z_{p^s}[x]$. Clearly, it preserves the identity. Consider now an equality 
\beq
\label{staps1}
\p_{(i_n)_n,(0)_n,1}\p_{(i'_n)_n,(0)_n,1} = \p_{(i''_n)_n,(0)_n,1}
\eeq
in $\Phi_m$. We must show that 
$$\left(\sum_{k \in \N} i_{mk} x^k\right)\left(\sum_{k \in \N} i'_{mk} x^k\right) = \sum_{k \in \N} i''_{mk} x^k,$$
which is equivalent to
$$\sum_{k_1,k_2 \in \N} i_{mk_1}i'_{mk_2} x^{k_1+k_2} = \sum_{k \in \N} i''_{mk} x^k$$
and therefore to
\beq
\label{staps2}
\sum_{j = 0}^k i_{mj}i'_{mk-mj} = i''_{mk}
\eeq
holding for all $k \in \N$.

On the other hand, it follows from (\ref{autlg9}) and (\ref{staps1}) that
\beq
\label{staps3}
i''_{mk} = \sum_{\ell \in \Z} i'_{mk-\ell}i_{\ell}.
\eeq
Since $i_n = i'_n = 0$ if $\frac{n}{m} \notin \N$, (\ref{staps3}) implies (\ref{staps2}) and so $\Lambda$ is a monoid homomorphism.

We prove next that
\beq
\label{staps4}
\delta_{m,pr}^{p^s} = 1.
\eeq

Since $\Lambda$ is an injective monoid homomorphism, it suffices to show that $(\delta_{m,pr}\Lambda)^{p^s} = 1$, i.e. $(1+prx)^{p^s} = 1$. In view of Newton's binomial theorem, this is equivalent to 
$$\sum_{j=1}^{p^s} \binom{p^s}{j} (prx)^j = 0.$$
Thus it suffices to show that 
\beq
\label{staps5}
p^s \mid \binom{p^s}{j} p^j
\eeq
for $j = 1,\ldots,p^s$. 

Given $n \in \Z \setminus \{ 0 \}$, let 
$$n\nu_p = \max\{ k \in \N \mid p^k \mbox{ divides }n\}.$$
Write also $0\nu_p = \infty$. 
Given $j \in \{ 1,\ldots,p^s\}$, we can write
\beq
\label{binomial-product-p}
\binom{p^s}{j} p^j = p^s\frac{p^s-1}{1}\frac{p^s-2}{2}\ldots \frac{p^s-(j-1)}{j-1}\frac{p^j}{j}.
\eeq
Since $(p^s-i)\nu_p = i\nu_p$ for $i = 1,\ldots,j-1$, then $p$ does not divide neither the numerator
nor the denominator of $\frac{p^s-i}{i}$, and since $j\nu_p < j$, then $p$ divides the numerator
but not the denominator of $\frac{p^j}{j}$. Since both sides of equation (\ref{binomial-product-p})
represent a positive integer, then (\ref{staps5}) holds and so does (\ref{staps4}). Therefore $\delta_{m,pr} \in \stab_{\L_{p^s}}(\xi)$. 

So far, we have stablished that
$${\rm Stab}_{\L_{p^s}}(\xi) \subseteq \{ \p_{(i_n)_n,(0)_n,1} \mid \mbox{ there exists a unique $m \in \Z$ such that $p$ does not divide }i_m\}$$
and
$$\langle \lambda, \eta_j, \delta_{m,pr} \mid j \in \Z^*_k,\, m \in \Z \setminus \{ 0\},\, r \in \Z_{p^s} \rangle \subseteq {\rm Stab}_{\L_{p^s}}(\xi).$$
To complete the proof of the lemma, we show that, for every $\p_{(i_n)_n,(0)_n,1}$ such that there exists a unique $m \in \Z$ with $p$ not dividing $i_m$, there exists some $\theta \in \langle \lambda, \eta_j, \delta_{m,pr} \mid j \in \Z^*_k,\, m \in \Z \setminus \{ 0\},\, r \in \Z_{p^s} \rangle$ such that $\p_{(i_n)_n,(0)_n,1}\theta = 1$. Replacing $\p_{(i_n)_n,(0)_n,1}$ by $\p_{(i_n)_n,(0)_n,1}\lambda^{-m}\eta_{i_m\inv}$ if needed, we may assume that $i_0 = 1$ from now on.

We use a double induction scheme. Let $\psi = \p_{(i_n)_n,(0)_n,1}$ and define
$$t_{\psi} = \min\{ i_n\nu_p \mid n \in \Z \setminus \{ 0\}\}, \quad u_{\psi} = |\{ n \in \Z \setminus \{ 0\} \mid i_n\nu_p = t_{\psi} \}|.$$

We start by considering the case $t_{\psi} \geq \frac{s}{2}$. Take $r \in \Z$ such that $i_r\nu_p = t_{\psi}$ and let $\psi' = \psi\delta_{r,-i_r}$. 
By (\ref{autlg9}), we have
$$\psi' = \p_{(i_{n} -i_r i_{n-r})_n,(0)_n,1}.$$
Let $s' = \lceil \frac{s}{2} \rceil$. If $n-r \neq 0$, then $p^{s'}$ divides both $i_r$ and $i_{n-r}$, hence $i_r i_{n-r} = 0$. But if $n = r$ then $i_n - i_r i_{n-r} = 0$, hence $\psi'$ is obtained from $\psi$ by replacing $i_r$ by 0. Applying successively this procedure, we end up by obtaining the identity mapping, hence the claim holds in the case $t_{\psi} \geq \frac{s}{2}$. 

We consider now the case $t_{\psi} = t < \frac{s}{2}$ and we assume that the claim holds for all $\psi'$ such that $t_{\psi'} > t$. Take $r \in \Z$ such that $i_r\nu_p = t_{\psi}$ and let $\psi' = \psi\delta_{r,-i_r}$. 
By (\ref{autlg9}), we have
$$\psi' = \p_{(i_{n} -i_r i_{n-r})_n,(0)_n,1}.$$
If $n = r$ then $i_n - i_r i_{n-r} = 0$. If $n \neq r$, then 
$$(i_{n} -i_r i_{n-r})\nu_p \geq \min\{ (i_{n}\nu_p, (i_r i_{n-r})\nu_p \} = \min\{ (i_{n}\nu_p, i_r\nu_p + i_{n-r}\nu_p \}
\geq \min\{ (i_{n}\nu_p, 2t_{\psi} \},$$
hence either $\psi' = 1$, or $t_{\psi'} > t_{\psi}$, or $t_{\psi'} = t_{\psi}$ and $u_{\psi'} < u_{\psi}$. 
Therefore the claim follows by induction.
\end{proof}

\begin{lemma}
\label{stanfg}
Let $p$ be a positive prime and $s \geq 2$. Then ${\rm Stab}_{\L_{p^s}}(\xi)$ is not finitely generated.
\end{lemma}

\begin{proof}
Suppose that ${\rm Stab}_{\L_{p^s}}(\xi)$ is finitely generated. In view of Lemma \ref{staps}, it admits a generating set of the form $A_M$, where
$$A_M = \{ \lambda, \eta_j, \delta_{m,pr} \mid j \in \Z^*_k,\, -M \leq m \leq M,\, r \in \Z_{p^s} \},$$
for some $M \geq 1$. 

Since ${\rm Stab}_{\L_{p^s}}(\xi)$ is abelian by (\ref{autlg8}), we may write $\delta_{Ms,p} = \psi\lambda^t$ for some $\psi \in \langle A_M \setminus \{ \lambda \} \rangle$ and $t \in \Z$. Thus
$$\alpha\psi = \alpha\delta_{Ms,p}\lambda^{-t} = \xi^{-t}(\alpha\beta_{Ms}^p)\xi^t = \beta_{-t}\beta_{Ms-t}^{p}.$$
Since the set of automorphisms $\p_{(i_n)_n,(0)_n,1}$ satisfying $i_0 \in \Z^*_{p^s}$ contains $A_M$ and is closed under composition in view of (\ref{autlg9}), then $t = 0$ and so 
$\delta_{Ms,p} \in \langle A_M \setminus \{ \lambda \} \rangle$. 

Using (\ref{autlg8}) once again, we may write 
$$\delta_{Ms,p} = \p_{(i_n^{(1)})_n,(0)_n,1}\ldots \p_{(i_n^{(q)})_n,(0)_n,1}$$
for some $\p_{(i_n^{(1)})_n,(0)_n,1}, \ldots, \p_{(i_n^{(q)})_n,(0)_n,1}$ (note that $A_M \setminus \{ \lambda \}$ is closed under inversion in view of
(\ref{staps4})). 
By (\ref{autlg10}), we get
$$\delta_{Ms,p} = \p_{(\sum_{k_1 + \ldots + k_q = n} i_{k_1}^{(1)} \ldots i_{k_q}^{(q)})_n,(0)_n,1},$$
hence comparing the $Ms$ components yields
$$p = \sum_{k_1 + \ldots + k_q = Ms} i_{k_1}^{(1)} \ldots i_{k_q}^{(q)}.$$
It follows that $i_{k_1}^{(1)} \ldots i_{k_q}^{(q)} \neq 0$ for some $k_1, \ldots , k_q \in \Z$ satisfying $k_1 + \ldots + k_q = Ms$. In particular, $i_{k_1}^{(1)}, \ldots, i_{k_q}^{(q)} \neq 0$. Given the structure of the elements of $A_M \setminus \{ \lambda \}$, it follows that $-M \leq k_1, \ldots,k_q \leq M$. Since $k_1 + \ldots + k_q = Ms$, there exist at least $s$ nonzero elements among $k_1, \ldots, k_q$ and so it follows that $p$ divides at least $s$ elements among $i_{k_1}^{(1)}, \ldots, i_{k_q}^{(q)}$, and so $i_{k_1}^{(1)} \ldots i_{k_q}^{(q)} = 0$, a 
contradiction. Therefore ${\rm Stab}_{\L_{p^s}}(\xi)$ is not finitely generated.
\end{proof}

We consider next the automorphisms of $\L_k$. Assume that
$\p_{(i_n)_n,(j_n)_n,r}$ is an automorphism of $\L_k$. Since
$\fin(L_k) \cup \{ \xi^r\}$ generates a proper subgroup of $\L_k$ unless $r
= \pm 1$, we can write
$$\aut(\L_k) = \aut_+(\L_k) \cup \aut_-(\L_k),$$
where $\aut_+(\L_k)$ denotes the set of automorphisms of the form $\p_{(i_n)_n,(j_n)_n,1}$ (the {\em positive} automorphisms) and $\aut_-(\L_k)$ denotes the set of automorphisms of the form $\p_{(i_n)_n,(j_n)_n,-1}$ (the {\em negative} automorphisms). Let $\zeta$ be the endomorphism of $\L_k$ defined by $\alpha\zeta = \alpha$ and $\xi\zeta = \xi\inv$. It follows that $\zeta \in \aut_-(\L_k)$. It is easy to check that $\aut_+(\L_k)$ is a (normal) subgroup of index 2 of $\aut(\L_k)$ containing 
${\rm Stab}_{\L_k}(\xi)$. Thus
\beq
\label{sd}
{\rm Aut}(\L_k) \cong {\rm Aut}_+(\L_k) \rtimes C_2.
\eeq

\begin{lemma}
\label{autst}
Let $k \geq 2$. Then
$$\begin{array}{rcl}
\sigma:{\rm Aut}_+(\L_k)&\to&{\rm Stab}_{\L_{k}}(\xi)\\
\p_{(i_n)_n,(j_n)_n,1}&\mapsto&\p_{(i_n)_n,(0)_n,1}
\end{array}$$
is a surjective group homomorphism. 
\end{lemma}

\begin{proof}
Let $\p_{(i_n)_n,(j_n)_n,1}, \p_{(i'_n)_n,(j'_n)_n,1} \in {\rm Aut}_+(\L_k)$. By (\ref{autlg4}), we have
$$\begin{array}{lll}
(\p_{(i_n)_n,(j_n)_n,1}\p_{(i'_n)_n,(j'_n)_n,1})\sigma&=&\p_{(\sum_{q \in \Z} i'_{n-q}i_{q})_n, (0)_n,1} = \p_{(i_n)_n,(0)_n,1}\p_{(i'_n)_n,(0)_n,1}\\
&=&(\p_{(i_n)_n,(j_n)_n,1}\sigma)(\p_{(i'_n)_n,(j'_n)_n,1}\sigma).
\end{array}$$
Hence $\sigma$ is a monoid endomorphism from ${\rm Aut}_+(\L_k)$ into the monoid of endomorphisms of $\L_k$ fixing $\xi$. If we take $\p_{(i'_n)_n,(j'_n)_n,1} = \p_{(i_n)_n,(j_n)_n,1}\inv$, it follows that $(\p_{(i_n)_n,(j_n)_n,1}\sigma)(\p_{(i'_n)_n,(j'_n)_n,1}\sigma) = 1$. Similarly, $(\p_{(i'_n)_n,(j'_n)_n,1}\sigma)(\p_{(i_n)_n,(j_n)_n,1}\sigma) = 1$. Therefore $\sigma$ is a (group) homomorphism from ${\rm Aut}_+(\L_k)$ into ${\rm Stab}_{\L_{k}}(\xi)$.
Since $\sigma$ fixes all the automorphisms from ${\rm Stab}_{\L_{k}}(\xi)$, it is surjective.
\end{proof}

\begin{lemma}
\label{buildaut}
Let $k \geq 2$ and $r \in \Z$. Let $(i_n)_n$ and $(j_n)_n$ be mappings with finite support  from $\Z$ to $\Z_k$. Then the following conditions are equivalent:
\bi
\item[(i)] $\p_{(i_n)_n,(j_n)_n,r} \in {\rm Aut}(\L_k)$;
\item[(ii)] $\p_{(i_n)_n,(0)_n,1} \in {\rm Stab}_{\L_k}(\xi)$ and $r = \pm 1$.
\ei
\end{lemma}

\begin{proof}
(i) $\Rw$ (ii). We have already remarked that (i) implies $r = \pm 1$. If $r = 1$, then $\p_{(i_n)_n,(0)_n,1} \in {\rm Stab}_{\L_k}(\xi)$ follows from Lemma \ref{autst}. Hence we may assume that $r = -1$. Composing with $\zeta = \p_{(\varepsilon_n)_n, (0)_n, -1}$, we get
$$\begin{array}{lll}
\alpha\p_{(i_n)_n,(j_n)_n,-1}\zeta&=&\left(\prod_{n \in \Z} \beta_n^{i_n}\right)\zeta = \prod_{n \in \Z} \beta_{-n}^{i_n} = \prod_{n \in \Z} \beta_n^{i_{-n}},\\
\xi\p_{(i_n)_n,(j_n)_n,-1}\zeta&=&\left(\left(\prod_{n \in \Z} \beta_n^{j_n}\right)\xi\inv\right)\zeta = \left(\prod_{n \in \Z} \beta_{-n}^{j_n}\right)\xi = \left(\prod_{n \in \Z} \beta_n^{j_{-n}}\right)\xi,
\end{array}$$
hence $\p_{(i_n)_n,(j_n)_n,-1}\zeta = \p_{(i_{-n})_n,(j_{-n})_n,1} \in \aut_+(\L_k)$ and so $\p_{(i_{-n})_n,(0)_n,1} \in {\rm Stab}_{\L_k}(\xi)$ by the preceding case. Let $\p_{(i'_{n})_n,(0)_n,1} = \p_{(i_{-n})_n,(0)_n,1}\inv$. It follows from (\ref{autlg4}) that
$$\sum_{k \in \Z} i'_{n-k}i_{-k} = \left\{
\begin{array}{ll}
1&\mbox { if }n = 0\\
0&\mbox{ otherwise }
\end{array}.
\right.$$
But then, for $i''_n = i'_{-n}$, we get
$$\sum_{k \in \Z} i''_{-n-k}i_{k} = \left\{
\begin{array}{ll}
1&\mbox { if }n = 0\\
0&\mbox{ otherwise }
\end{array}.
\right.$$
Replacing $-n-k$ by $n-k$, we get $\p_{(i''_{n})_n,(0)_n,1} = \p_{(i_{n})_n,(0)_n,1}\inv$ and so (ii) holds.

(ii) $\Rw$ (i). 
Assume first that $r = 1$. Write $u = \prod_{n \in \Z}\beta_n^{j_n}$ and $v = u\inv \p_{(i_{n})_n,(0)_n,1}\inv$. Let $\psi$ be the endomorphism of $\L_k$ defined by $\alpha\psi = \alpha\p_{(i_{n})_n,(0)_n,1}\inv$ and $\xi\psi = v\xi$. In view of (\ref{autlg3}), we get
$$\alpha\p_{(i_n)_n,(j_n)_n,1}\psi = \alpha\p_{(i_n)_n,(0)_n,1}\psi = \alpha$$
and
$$\xi\p_{(i_n)_n,(j_n)_n,1}\psi = (u\xi)\psi = (u\p_{(i_{n})_n,(0)_n,1}\inv)(v\xi) = \xi,$$
hence $\p_{(i_n)_n,(j_n)_n,1}\psi = 1$. Similarly, we can show that $\psi\p_{(i_n)_n,(j_n)_n,1} = 1$, hence $\p_{(i_n)_n,(j_n)_n,1} \in {\rm Aut}(\L_k)$.

Assume now that $r = -1$. We have
$\alpha\zeta\p_{(i_n)_n,(j_n)_n,-1} = \prod_{n \in \Z} \beta_n^{i_n}$ and 
$$\xi\zeta\p_{(i_n)_n,(j_n)_n,-1} = \xi\inv\p_{(i_n)_n,(j_n)_n,-1} = \xi\prod_{n \in \Z} \beta_n^{-j_n} = (\prod_{n \in \Z} \beta_{n+1}^{-j_n}) \xi = (\prod_{n \in \Z} \beta_{n}^{-j_{n-1}}) \xi.$$
Thus $\zeta\p_{(i_n)_n,(j_n)_n,-1} = \p_{(i_n)_n,(-j_{n-1})_n,1}$, which is an automorphism by the preceding case. Since $\zeta$ is an automorphism, so is 
$\p_{(i_n)_n,(j_n)_n,-1}$.
\end{proof}

\begin{lemma}
\label{embedding}
Let $u,v \geq 2$ be coprime integers. Then $\L_{u} \cong \langle \alpha^v,\xi \rangle \leq \L_{uv}$.
\end{lemma}

\begin{proof}
By (\ref{lapre}),
$$\langle \alpha,\xi \mid \alpha^{uv}, [\xi^m\alpha\xi^{-m}, \xi^n\alpha\xi^{-n}]\; (m,n \in \Z)\rangle$$
is a presentation  of $\L_{uv}$ and
$$\langle \alpha,\xi \mid \alpha^{u}, [\xi^m\alpha\xi^{-m}, \xi^n\alpha\xi^{-n}]\; (m,n \in \Z)\rangle$$
is a presentation  of $\L_{u}$.
Let $H = \langle \alpha^v,\xi \rangle \leq \L_{uv}$. We claim that $H = \langle \alpha^v,\xi\rangle \leq \L_{uv}$ is isomorphic to $\L_u$. 

Indeed, it follows from the provided presentations that
$$\alpha \mapsto \alpha^v,\quad \xi \mapsto \xi$$
defines a surjective homomorphism $\theta$ from $\L_u$ onto $H$. Suppose that $w = (\prod_{n \in \Z} \beta_n^{i_n})\xi^r \in \ker\theta$, where $(i_n)_n$ is a mapping from $\Z$ into $\Z_u$ with finite support and $r \in \Z$. Then $1 = w\theta = (\prod_{n \in \Z} \beta_n^{vi_n})\xi^r$ yields $r = 0$ and $uv \mid vi_n$ for every $n$. Thus $u \mid i_n$ for every $n$ and so $w = 1$. Therefore $\theta:\L_u \to H$ is an isomorphism.
\end{proof}

\begin{lemma}
\label{coprime}
Let $p_1,\ldots,p_t$ be distinct positive primes with $t \geq 2$. Let $c_1 \geq 2$ and $c_2,\ldots, c_t \geq 1$. Then ${\rm Stab}_{\L_{p_1^{c_1}}}(\xi)$ is a homomorphic image of
${\rm Stab}_{\L_{p_1^{c_1}\ldots p_t^{c_t}}}(\xi)$.
\end{lemma}

\begin{proof}
Write $u =  p_1^{c_1}$, $v = p_2^{c_2}\ldots p_t^{c_t}$ and $k=uv$.
Let $H = \langle \alpha^v,\xi \rangle \leq \L_{uv}$. By Lemma \ref{embedding}, we have $H \cong \L_u$. 

We define
$$\begin{array}{rcl}
\omega:{\rm Stab}_{\L_{k}}(\xi)&\to&{\rm Stab}_{H}(\xi)\\
\psi&\mapsto&\psi|_H.
\end{array}$$
We show that $H\psi \subseteq H$ for each $\psi \in {\rm Stab}_{\L_{k}}(\xi)$. Indeed, we may assume that $\alpha\psi = \prod_{n \in \Z} \beta_n^{i_n}$ for some mapping $(i_n)_n$ from $\Z$ into $\Z_k$ with finite support. Since $\xi\psi = \xi$ and $\alpha^v\psi = \prod_{n \in \Z} \beta_n^{vi_n}$, we get $H\psi \subseteq H$. This implies that $\omega$ is a (group) homomorphism.

We show next that $\omega$ is surjective. In view of Lemma \ref{staps}, it suffices to show that
$$\{ \lambda, \eta_j, \delta_{m,pr} \mid j \in \Z^*_k,\, m \in \Z \setminus \{ 0\},\, r \in \Z_{p^s} \} \subseteq {\rm Im} \; \omega$$
(after translation from $\L_u$ to $H$).

Clearly, $\lambda$ in $H$ (defined by $\alpha^v \mapsto \xi\alpha^v\xi\inv$ and $\xi \mapsto \xi$) is the restriction of $\lambda$ from $\L_k$.

Consider next $\eta_j$ in $H$ (defined by $\alpha^v \mapsto (\alpha^v)^j$ and $\xi \mapsto \xi$, where $j \in \Z^*_u$).  
We can represent $j$ as an integer not divisible by $p_1$. Let $v'$ be the product of the primes $p_2,\ldots,p_t$ which do not divide $j$ and set $j' = j+p_1^{c_1}v'$. 

Since $j \in \Z^*_u$, then $p_1$ does not divide $j$ and so it does not divide $j'$. Suppose that $p_i$ divides $j'$ for some $2 \leq i \leq t$. If $p_i$ does not divide $j$, then it divides $v'$ and so it divides $j = j' - p_1^{c_1}v'$, a contradiction. If $p_i$ divides $j$, then it divides $p_1^{c_1}v' = j'-j$, and so must be a factor of $v'$, yet a contradiction. 

Thus we can consider $j' \in \Z^*_k$ and $\eta_{j'} \in \stab_{\L_{k}}(\xi)$. We claim that $\eta_{j'}\omega = \eta_j$. Since $\xi\eta_{j'} = \xi = \xi\eta_j$, it remains to be shown that $\alpha^v\eta_{j'} = \alpha^v\eta_j$. This is equivalent to show that $\alpha^{vj'} = \alpha^{vj}$ in $\L_k$, i.e. $k$ divides $vj'-vj$, i.e. $p_1^{c_1}$ divides $j'-j$, which is clearly true.

Before dealing with the remaining case, we remark that the following generalization of (\ref{staps4}) holds with the same proof
\beq
\label{coprime1}
\mbox{if $p = p_1\ldots p_t$ and $c = \max\{ c_1,\ldots,c_t \}$, then $\delta_{m,pr}^{p^c} = 1$.}
\eeq
Indeed, we adapt $\Phi_m$ and $\Lambda$ in the obvious way, and everything works the same.

Now consider $\delta_{m,p_1r}$ in $H$ (defined by $\alpha^v \mapsto \alpha^v\beta_m^{vp_1r}$ and $\xi \mapsto \xi$). Since $p_1^{c^1-1}$ and $p_2\ldots p_t$ are coprime, there exist $r',y \in \Z$ such that $p_2\ldots p_tr' + p_1^{c_1-1}y = r$. Writing $p = p_1\ldots p_t$, we consider now $\delta_{m,pr'}$, which is in $\stab _{\L_{k}}(\xi)$ by (\ref{coprime1}). We claim that $\delta_{m,pr'}\omega = \delta_{m,p_1r}$. Since $\xi\delta_{m,pr'} = \xi = \xi\delta_{m,p_1r}$, it remains to be shown that $\alpha^v\delta_{m,pr'} = \alpha^v\delta_{m,p_1r}$. This is equivalent to show that $\alpha^{v}\beta_m^{vpr'} = \alpha^{v}\beta_m^{vp_1r}$ in $\L_k$. So it suffices to show that $k$ divides $vpr' - vp_1r$, i.e. $p_1^{c_1-1}$ divides $p_2\ldots p_tr' -r$, which is certainly true. 
Therefore $\omega$ is surjective as claimed. 
\end{proof}

\begin{lemma}
\label{stdis}
Let $k = p_1\ldots p_s$, where $p_1,\ldots,p_s$ are distinct positive primes and $s \geq 2$. For $\ell = 1,\ldots,s$, let
$\rho_{\ell}$ be the endomorphism of $\L_k$ defined by $\alpha\rho_{\ell} = \alpha^{p_{\ell}}\beta_1^{k/p_{\ell}}$ and $\xi\rho_{\ell} = \xi$.
Then:
\bi
\item[(i)]
${\rm Stab}_{\L_{k}}(\xi) = \langle \eta_j, \rho_{\ell} \mid j \in \Z^*_k,\, \ell \in \{ 1,\ldots,s\} \rangle;$
\item[(ii)]
$\langle \lambda, \eta_j, \rho_{\ell} \mid j \in \Z^*_k,\, \ell \in \{ 1,\ldots,s-1\} \rangle$ is a finite index subgroup of ${\rm Stab}_{\L_{k}}(\xi)$.
\ei
\end{lemma}

\begin{proof}
(i) Let $\ell \in \{ 1,\ldots,s \}$. We must show that $\rho_{\ell}$ is an automorphism.

Write $k_{\ell} = \frac{k}{p_{\ell}}$. Since $p_{\ell}^2$ and $k_{\ell}^2$ are coprime, there exist $u_{\ell},v_{\ell} \in \Z$ such that $p_{\ell}^2u_{\ell} + k_{\ell}^2v_{\ell} = 1$. Let $\psi$ denote the endomorphism of $\L_k$ defined by $\alpha\psi = \alpha^{p_{\ell}u_{\ell}}\beta_{-1}^{k_{\ell}v_{\ell}}$ and $\xi\psi = \psi$. It follows from (\ref{autlg9}) that
$$\rho_{\ell}\psi = \p_{(i'_{0}i_{n} + i'_{-1}i_{n+1})_n,(0)_n,1}.$$
\bi
\item
If $n = -1$, then $i'_{0}i_{n} + i'_{-1}i_{n+1} = k_{\ell}v_{\ell}p_{\ell} = 0$ in $\Z_k$.
\item
If $n = 0$, then $i'_{0}i_{n} + i'_{-1}i_{n+1} = p_{\ell}u_{\ell}p_{\ell} + k_{\ell}v_{\ell}k_{\ell} = 1$.
\item
If $n = 1$, then $i'_{0}i_{n} + i'_{-1}i_{n+1} = p_{\ell}u_{\ell}k_{\ell} = 0$ in $\Z_k$.
\ei
Since $i'_{0}i_{n} + i'_{-1}i_{n+1} = 0$ for all the remaining values of $n$, we get $\rho_{\ell}\psi = 1$. In view of (\ref{autlg8}), $\rho_{\ell} \in {\rm Stab}_{\L_{k}}(\xi)$.

We show next that 
\beq
\label{stdis1}
\alpha\rho_{\ell}^m = \alpha^{p^m_{\ell}}\beta_m^{k^m_{\ell}} \mbox{ for every }m \geq 1.
\eeq

We use induction on $m$. The case $m = 1$ holds by definition. Assume that the claim holds for some $m \geq 1$. In view of (\ref{autlg11}), we get
$$\alpha\rho_{\ell}^{m+1} = \alpha^{p^{m+1}_{\ell}}\beta_m^{k^m_{\ell}p_{\ell}}\beta_1^{p^m_{\ell}k_{\ell}}\beta_{m+1}^{k^{m+1}_{\ell}} = \alpha^{p^{m+1}_{\ell}}\beta_{m+1}^{k^{m+1}_{\ell}}$$
where the last equality follows from the fact that $k \mid k^m_{\ell}p_{\ell}$ and $k \mid
p^m_{\ell}k_{\ell}$ and so $\beta_m^{k^m_{\ell}p_{\ell}}\beta_1^{p^m_{\ell}k_{\ell}}=1$ and (\ref{stdis1}) holds.

Similarly, we can show that
\beq
\label{stdis2}
\alpha\rho_{\ell}^{-m} = \alpha^{p^m_{\ell}u^m_{\ell}}\beta_{-m}^{k^m_{\ell}v^m_{\ell}} \mbox{ for every }m \geq 1.
\eeq

In view of (\ref{stdis1}), and since $k$ divides $k_ik_j$ whenever $i \neq j$, a straightforward induction shows that
\beq
\label{aves}
\alpha\rho_1^{m_1} \ldots \rho_r^{m_r} = \alpha^{p_1^{m_1}\ldots p_r^{m_r}}
\prod_{i=1}^r \beta_{m_i}^{k_i^{m_i}\prod_{\ell \in \{ 1,\ldots r\} \setminus \{ i\} } p_{\ell}^{m_{\ell}}}
\eeq
holds for all $r \in \{ 1,\ldots,s\}$ and $m_1,\ldots,m_r > 0$.

Now let $\psi \in {\rm Stab}_{\L_{k}}(\xi)$. We prove that, for each $\ell \in \{ 1,\ldots,s\}$,
\beq
\label{stdis3}
\alpha^{k_{\ell}}\psi = \beta_{m_{\ell}}^{k_{\ell}j_{\ell}}  \mbox{ for some $m_{\ell} \in \Z$ and }j_{\ell} \in \Z^*_{\ell}.
\eeq

Indeed, if $H = \langle \alpha^{k_{\ell}},\xi \rangle$, it follows from Lemma \ref{embedding} that $H \cong \L_{p_{\ell}}$. Since $\alpha^{k_{\ell}}\psi = (\alpha\psi)^{k_{\ell}}$, we get $H\psi \subseteq H$. Also $H\psi\inv \subseteq H$ out of symmetry, hence $H = H\psi\inv \psi \subseteq H\psi \subseteq H$, yielding $H\psi = H$. Hence $\psi|_H \in \stab_H(\xi)$. Since $H \cong \L_{p_{\ell}}$, then (\ref{stdis3}) follows from Lemma \ref{stap}.

Now the greatest common divisor of $k_1,\ldots,k_s$ is obviously 1, so there exist $x_1,\ldots,x_s \in \Z$ such that $k_1x_1 + \ldots + k_sx_s = 1$. It follows from (\ref{stdis3}) that
\beq
\label{empty-case-s}
\alpha\psi = \beta_{m_{1}}^{k_{1}j_{1}x_1}\ldots \beta_{m_{s}}^{k_{s}j_{s}x_s}.
\eeq
Note that the $m_{\ell}$ are not necessarily nonzero.
We show that, for every $\ell \in \{ 1,\ldots,s+1\}$
\beq
\label{stdis4}
\mbox{$\alpha\psi\rho_s^{-m_s}\ldots \rho_{\ell}^{-m_{\ell}} = \beta_{m_1}^{k_1y_1} \ldots \beta_{m_{\ell -1}}^{k_{\ell-1}y_{\ell-1}} \alpha^{k_{\ell}y_{\ell}}\ldots \alpha^{k_{s}y_{s}}$ for some $y_1,\ldots,y_s \in \Z$}.
\eeq

The formula for $\ell = s+1$ is to be interpreted as computing $\alpha \psi$ and so it is trivially
true because of (\ref{empty-case-s}). Assume equation (\ref{stdis4}) holds for $\ell \in \{ 2,\ldots,s+1\}$. If $m_{\ell-1} \neq 0$, it follows from (\ref{stdis1}) or (\ref{stdis2}) that there exist some $z,w \in \Z$ such that 

\beq
\label{auxiliary}
\alpha\rho_{\ell-1}^{-m_{\ell-1}} = \alpha^{p_{\ell-1}z}\beta_{-m_{\ell-1}}^{k_{\ell-1}w}.
\eeq

This is in fact also true if $m_{\ell-1} = 0$ because $p_{\ell-1}$ and $k_{\ell-1}$ are coprime. We apply now (\ref{autlg11}) to the automorphisms $\psi\rho_s^{-m_s}\ldots \rho_{\ell}^{-m_{\ell}}$ and $\rho_{\ell-1}^{-m_{\ell-1}}$ and using (\ref{stdis4}) and (\ref{auxiliary}). Since $k = p_{\ell-1}k_{\ell-1}$ and it divides the product of any two distinct $k_i$, we get
$$\alpha\psi\rho_s^{-m_s}\ldots \rho_{\ell-1}^{-m_{\ell-1}} = \beta_{m_1}^{k_1y_1p_{\ell-1}z} \ldots \beta_{m_{\ell -2}}^{k_{\ell-2}y_{\ell-2}p_{\ell-1}z} \alpha^{k_{\ell-1}^2y_{\ell-1}w}\alpha^{k_{\ell}y_{\ell}p_{\ell-1}z}\ldots \alpha^{k_{s}y_{s}p_{\ell-1}z},$$
proving (\ref{stdis4}). In particular, for $\ell = 1$, we get 
$$\alpha\psi\rho_s^{-m_s}\ldots \rho_{1}^{-m_{1}} = \alpha^j$$ for some $j \in \Z_k$. Since automorphisms preserve order, we must have $j \in \Z^*_k$, hence $\alpha\psi\rho_s^{-m_s}\ldots \rho_{1}^{-m_{1}} = \alpha\eta_j$, yielding 
$$\psi = \eta_j\rho_1^{m_1}\ldots \rho_{s}^{m_{s}}$$
and we are done.

(ii) We show that 
\beq
\label{aves2}
\rho_1^m\ldots \rho_{s}^m = \lambda^m
\eeq
for $m = (p_1-1)\ldots(p_s-1)$. 

Indeed, it follows from (\ref{aves}) that 
$$\alpha\rho_1^{m} \ldots \rho_s^{m_s} = \alpha^{k^m}
\prod_{i=1}^s \beta_{m}^{k_i^{2m}} = \beta_m^{\sum_{i=1}^s k_i^{2m}}.$$
Thus it suffices to show that $k$ divides $\sum_{i=1}^s k_i^{2m} -1$. This is equivalent to have $p_i$ dividing $k_i^{2m} -1$ for $i = 1,\ldots,s$. Since $p_i$ does not divide $k_i$, we have $k_i^{p_i-1} \equiv 1\, ({\rm mod}\, p_i)$, hence $k_i^{2m} \equiv 1\, ({\rm mod}\, p_i)$. Therefore (\ref{aves2}) holds.

Since ${\rm Stab}_{\L_{k}}(\xi)$ is abelian, it follows from part (i) and (\ref{aves2}) that
$${\rm Stab}_{\L_{k}}(\xi) = \bigcup_{i=0}^{m-1} \langle \lambda, \eta_j, \rho_{\ell} \mid j \in \Z^*_k,\, \ell \in \{ 1,\ldots,s-1\} \rangle \rho_s^i.$$
\end{proof}


In view of Lemma \ref{buildaut}, we can define $\Psi_k$ to be the set of all automorphisms of $\L_k$ of the form $\p_{(\varepsilon_{n-r})_n,(j_n)_n,1}$, where $r \in \Z$ and $(j_n)_n$ is a mapping from $\Z$ to $\Z_r$. 

\begin{lemma}
\label{psik}
For every $k \geq 2$, $\Psi_k \unlhd {\rm Aut}(\L_k)$ and $\Psi_k \cong \L_k$.
\end{lemma}

\begin{proof}
Every element $x \in \L_k$ can be uniquely written in the form
$$x = \xi^r \left(\prod_{n\in \Z} \beta_{n}^{j_n}\right),$$
where $r \in \Z$ and $(j_n)_n$ is a mapping from $\Z$ to $\Z_r$. 

Let $\mu:\L_k \to \Psi_k$ be defined by
$$\left(\xi^r \left(\prod_{n\in \Z} \beta_{n}^{j_n} \right) \right)\mu = \p_{(\varepsilon_{n+r})_n,(j_n)_n,1}.$$
It is clear that $\mu$ is a bijection. 

Assume now that $x = \xi^r \left(\prod_{n\in \Z} \beta_{n}^{j_n} \right)$ and
$x' = \xi^{r'}\left(\prod_{n\in \Z} \beta_{n}^{j'_n}\right)$ are arbitrary elements of $\L_k$.
We have
$$xx' = \xi^r\left(\prod_{n\in \Z} \beta_{n}^{j_n}\right)\xi^{r'}\left(\prod_{n\in \Z} \beta_{n}^{j'_n}\right) = \xi^{r+r'}\left(\prod_{n\in \Z} \beta_{n-r'}^{j_n}\right)\left(\prod_{n\in \Z} \beta_{n}^{j'_n}\right) = \xi^{r+r'}\left(\prod_{n\in \Z} \beta_{n}^{j_{n+r'}+j'_n}\right),$$
hence
$$(xx')\mu = \left(\xi^{r+r'}\left(\prod_{n\in \Z} \beta_{n}^{j_{n+r'}+j'_n}\right)\right)\mu =
\p_{(\varepsilon_{n+r+r'})_n,(j_{n+r'}+j'_n)_n,1}.$$
On the other hand, using (\ref{autlg3}), we have
$$\begin{array}{lll}
\alpha(x\mu)(x'\mu)&=&\alpha\p_{(\varepsilon_{n+r})_n,(j_n)_n,1}\p_{(\varepsilon_{n+r'})_n,(j'_n)_n,1} = \beta_{-r}\p_{(\varepsilon_{n+r'})_n,(j'_n)_n,1} = 
\prod_{n \in \Z} \beta_{n}^{\varepsilon_{n+r'+r}}\\
&=&\alpha\p_{(\varepsilon_{n+r+r'})_n,(j_{n+r'}+j'_n)_n,1} = \alpha((xx')\mu),\\
\xi(x\mu)(x'\mu)&=&\xi\p_{(\varepsilon_{n+r})_n,(j_n)_n,1}\p_{(\varepsilon_{n+r'})_n,(j'_n)_n,1} = \left(\left(\prod_{n\in \Z} \beta_n^{j_n}\right)\xi \right)\p_{(\varepsilon_{n+r'})_n,(j'_n)_n,1}\\
&=&\left(\prod_{n\in\Z}\prod_{m\in \Z}\beta_m^{\varepsilon_{m+r'-n}j_n}\right)\left(\prod_{n\in \Z} \beta_n^{j'_n}\right)\xi = \left(\prod_{n\in\Z}\beta_{n-r'}^{j_n}\right)\left(\prod_{n\in \Z} \beta_n^{j'_n}\right)\xi\\
&=&\left(\prod_{n\in\Z}\beta_{n}^{j_{n+r'}+j'_n}\right)\xi
= \xi\p_{(\varepsilon_{n+r+r'})_n,(j_{n+r'}+j'_n)_n,1} = \xi((xx')\mu)
\end{array}$$
and so $(x\mu)(x'\mu) = (xx')\mu$. Therefore $\mu$ is a group isomorphism and $\Psi_k$ is obviously a subgroup of $\aut(\L_k)$. 

It remains to show that $\Psi_k$ is normal. First we note that, for every $x \in \L_k$,
\beq
\label{conju}
\mbox{$x$ and $\alpha$ are conjugate if and only if $x = \beta_r$ for some $r \in \Z$.}
\eeq
Indeed, suppose that $x = y\alpha y\inv$ for some $y \in \L_k$. We may write $y = u\xi^r$ for some $u \in \fin(\L_k)$ and $r \in \Z$. Since $\fin(\L_k)$ is abelian, it follows that
$$x = y\alpha y\inv = u\xi^r\alpha\xi^{-r}u\inv = \xi^r\alpha\xi^{-r} = \beta_r.$$
The converse implication holds trivially. Since $\alpha\p_{(\varepsilon_{n-r})_n,(j_n)_n,1}=\beta_r$, it follows that
\beq
\label{conju2}
\Psi_k = \{ \psi \in \aut_+(\L_k) \mid \alpha\psi \mbox{ and $\alpha$ are conjugate in }\L_k \}.
\eeq
Let $\p \in \aut(\L_k)$ and $\psi \in \Psi_k$. Clearly, $\p\psi\p\inv \in \aut_+(\L_k)$. If $\psi = \p_{(\varepsilon_{n-r})_n,(j_n),1}$, it follows from (\ref{autlg3}) that
\beq
\label{tail}
\beta_m\psi = \prod_{n \in \Z} \beta_{n}^{\varepsilon_{n-r-m}} = \beta_{r+m} = \xi^r\beta_m\xi^{-r}
\eeq
holds for every $m \in \Z$, hence $u\psi = \xi^ru\xi^{-r}$ for every $u \in \fin(\L_k)$. In particular, for $u = \alpha\p$, we get
$$\alpha\p\psi\p\inv = (\xi^ru\xi^{-r})\p\inv = (\xi^r\p\inv)\alpha(\xi^r\p\inv)\inv.$$
Thus $\p\psi\p\inv \in \Psi_k$ by (\ref{conju2}) and so $\Psi_k \unlhd {\rm Aut}(\L_k)$.
\end{proof}

We now aim at producing a generating set for $\aut(\L_k)$. Let $\iota = \p_{(\varepsilon_n)_n,(\varepsilon_n)_n,1}$. 

\begin{lemma}
\label{gene}
For every $k \geq 2$, 
\bi
\item[(i)] ${\rm Aut}_+(\L_k) = \langle {\rm Stab}_{\L_k}(\xi) \cup \{ \iota \} \rangle$;
\item[(ii)] ${\rm Aut}(\L_k) = \langle {\rm Stab}_{\L_k}(\xi) \cup \{ \iota, \zeta \} \rangle$.
\ei
\end{lemma}

\begin{proof}
(i) Write $S_k = {\rm Stab}_{\L_k}(\xi)$ and let $\p = \p_{(i_n)_n,(j_n)_n,1} \in \aut_+(\L_k)$. By Lemma \ref{psik}, we can take $\psi = \p_{(\varepsilon_n)_n,(-j_n)_n,1} \in \Psi_k \leq \aut_+(\L_k)$. By (\ref{autlg3}) we have
\beq
\label{xifixed}
\xi \p \psi  
=\left( \left( \prod_{m \in \mathbb{Z}} (\beta_m \psi)^{j_m} \right) \xi \right) \psi =
\left( \left( \prod_{m \in \mathbb{Z}} \beta_m^{j_m} \right) \xi \right) \psi =
 \left( \prod_{m \in \mathbb{Z}} \beta_m^{j_m} \right) \left( \prod_{m \in \mathbb{Z}} \beta_m^{-j_m} \right) \xi=\xi
\eeq
and so we get $\p\psi \in S_k$, hence $\p \in \Psi_kS_k$. Since $\L_k = \langle \alpha,\xi\rangle$, it follows from Lemma \ref{psik} that 
\beq
\label{gene2}
\Psi_k = \langle \alpha\mu,\xi\mu\rangle = \langle \iota,\lambda \rangle.
\eeq
Since $\lambda \in S_k$, we get ${\rm Aut}_+(\L_k) = \langle S_k \cup \{ \iota \} \rangle$.

(ii) follows from the fact that $\aut_+(\L_k)$ is a subgroup of index 2 of $\aut_(\L_k)$ and $\zeta \in \aut_-(\L_k)$.
\end{proof}

Let ${\rm FStab}_{\L_k}(\xi)$ denote the set of all finite order automorphisms of $\L_k$ fixing $\xi$. Since ${\rm Stab}_{\L_k}(\xi)$ is abelian by 
(\ref{autlg8}), then ${\rm FStab}_{\L_k}(\xi)$ is an abelian group itself.

%

We say that $k \geq 2$ is {\em squarefree} if there exists no prime $p \in \N$ such that $p^2$ divides $k$.

\begin{theorem}
\label{sfree}
Let $k \geq 2$ have positive prime divisors $p_1,\ldots,p_s$. Then the following conditions are equivalent:
\bi
\item[(i)]
$k$ is squarefree;
\item[(ii)]
${\rm Aut}(\L_k)$ is finitely generated;
\item[(iii)]
${\rm Aut}_+(\L_k)$ is a finite extension of $\L_k \rtimes (\Z^{s-1} \times \Z_k^*)$.
\ei
If in addition $p_i$ divides $(\frac{k}{p_i})^2 -1$ for $i = 1,\ldots,s$, then
$${\rm Aut}_+(\L_k) \cong \L_k \rtimes (\Z^{s-1} \times \Z_k^*) \; \mbox{ and }\; 
{\rm Aut}(\L_k) \cong (\L_k \rtimes (\Z^{s-1} \times \Z_k^*)) \rtimes C_2.$$
\end{theorem}


\begin{proof}
(i) $\Rw$ (iii). 
By Lemma \ref{psik}, we have $\Psi_k \unlhd \aut_+(\L_k)$. 
Suppose first that $s = 1$ (so $k$ is prime). Define
$$A_k = \{ \eta_j \mid j \in \Z^*_k \}.$$
It is straightforward to check that
$$\ba{rcl}
\Z^*_k&\to&A_k\\
j&\mapsto&\eta_j
\ea$$
is a group isomorphism (for arbitrary $k$, in fact). Moreover, $\Psi_k \cap A_k = \{ id \}$. Let $\p \in \aut_+(\L_k)$. By Lemma \ref{stap}, $\alpha\p = \beta_r^j$ for some $r \in \Z$ and $j \in \Z^*_k$. Hence $\p\eta_{j\inv} \in \Psi_k$ and so $\p \in \Psi_kA_k$. Thus  
$\aut_+(\L_k) = \Psi_kA_k$ and so 
$${\rm Aut}_+(\L_k) \cong \Psi_k \rtimes A_k \cong \L_k \rtimes \Z^*_k$$ 
by Lemma \ref{psik}. Therefore ${\rm Aut}(\L_k) \cong (\L_k \rtimes (\Z^{s-1} \times \Z_k^*)) \rtimes C_2$ by (\ref{sd}).

Assume now that $s \geq 2$. 
We define
$$A_k = \langle \eta_j, \rho_{\ell} \mid j \in \Z^*_k,\, \ell \in \{ 1,\ldots,s-1\} \rangle.$$
Since $\stab_{\L_k}(\xi)$ is abelian by (\ref{autlg8}), Lemma \ref{stdis} implies that
every element of $\stab_{\L_k}(\xi)$ can be written in the form $\rho_1^{a_1}\ldots \rho_s^{a_s}\eta_j$ for some $a_1,\ldots,a_s \in \Z$ and $j \in \Z^*_k$. We show that
\beq
\label{sfree1}
\rho_1^{a_1}\ldots \rho_s^{a_s}\eta_j = 1 \mbox{ implies $a_1 = \ldots = a_s = 0$ and $j = 1$.}
\eeq
Let 
$$I_+ = \{ i \in \{ 1,\ldots,r\} \mid a_i > 0\}, \quad I_- = \{ i \in \{ 1,\ldots,r\} \mid a_i < 0\}.$$
Suppose that $I_+ \cup I_- \neq \emptyset$. Replacing by the inverses if needed, we may assume that $I_+ \neq \emptyset$. Since $\stab_{\L_k}(\xi)$ is abelian, $ \rho_1^{a_1}\ldots \rho_s^{a_s}\eta_j = 1$ implies
$$\alpha\eta_j\prod_{i \in I_-} \rho_i^{-a_i} = \alpha\prod_{i \in I_+} \rho_i^{a_i}.$$
It follows from (\ref{aves}) that
$$\alpha^{j\prod_{i \in I_-} p_i^{-a_i}}\prod_{i\in I_-} \beta_{-a_i}^{jk_i^{-a_i}\prod_{\ell \in I_-\setminus \{ i\}} p_{\ell}^{-a_{\ell}}} = \alpha^{\prod_{i \in I_+} p_i^{a_i}}\prod_{i\in I_+} \beta_{a_i}^{k_i^{a_i}\prod_{\ell \in I_+\setminus \{ i\}} p_{\ell}^{a_{\ell}}}.$$
Comparing the powers of $\alpha$, we deduce that $k$ divides the difference
$$j\prod_{i \in I_-} p_i^{-a_i} - \prod_{i \in I_+} p_i^{a_i}.$$
In particular, by taking $\ell \in I_+$, it folllows that $p_{\ell}$ divides $j\prod_{i \in I_-} p_i^{-a_i}$, a contradiction. Thus $I_+ \cup I_- = \emptyset$, i.e. $a_1 = \ldots = a_s = 0$. Now $j = 1$ and (\ref{sfree1}) holds.

It follows from (\ref{sfree1}) that 
$$A_k \cong \langle \rho_1\rangle \times \ldots \times \langle \rho_{s-1} \rangle \times \{ \eta_j \mid j \in \Z^*_k \}.$$
We have already remarked that $ \{ \eta_j \mid j \in \Z^*_k \} \cong \Z^*_k$. Suppose that $\rho_{\ell}$ has finite order $m$. By (\ref{stdis1}), we get $\alpha^{p_{\ell}^m} = \alpha$, hence $k$ (and therefore $p_{\ell}$) divides $p_{\ell}^m -1$, a contradiction. Thus each $\rho_{\ell}$ has infinite order and so 
\beq
\label{aves4}
A_k \cong \Z^{s-1} \times \Z^*_k.
\eeq

Consider the subgroup $\Psi_kA_k$ of ${\rm Aut}_+(\L_k)$. In view of Lemma \ref{psik}, we have $\Psi_kA_k = A_k\Psi_k$. Suppose that $\psi \in A_k \cap \Psi_k$. Then $\psi = \p_{(\varepsilon_{n-r})_n,(0)_n,1} = \lambda^r$ for some $r \in \Z$.  Hence $\lambda^r = \rho_1^{a_1}\ldots \rho_{s-1}^{a_{s-1}}\eta_j$ for some $a_1,\ldots,a_{s-1} \in \Z$ and $j \in \Z^*_k$. Let $m = (p_1-1)\ldots (p_s-1)$. By (\ref{aves2}), we have
$$p_1^{mr}\ldots p_s^{mr} = \lambda^{mr} = \rho_1^{ma_1}\ldots \rho_{s-1}^{ma_{s-1}}\eta_j^m,$$
and so (\ref{sfree1}) yields $r = 0$. Thus $A_k \cap \Psi_k = \{ id \}$. Since $\Psi_k \unlhd \Psi_kA_k$, we get in view of Lemma \ref{psik} and (\ref{aves4})
\beq
\label{aves3}
\Psi_kA_k = (\Psi_k \rtimes A_k) \cong (\L_k \rtimes (\Z^{s-1} \times \Z^*_k)).
\eeq

Now we want to show that $\Psi_kA_k$ is a finite index subgroup of ${\rm Aut}_+(\L_k)$.
Let $\p = \p_{(i_n)_n,(j_n)_n,1} \in \aut_+(\L_k)$. Take $\psi = \p_{(\varepsilon_n)_n,(-j_n)_n,1} \in \Psi_k$. In view of (\ref{tail}) and (\ref{xifixed}), we have
$$\alpha\p\psi = \prod_{n \in \Z} \beta_n^{i_n}, \quad \xi\p\psi = \left(\left(\prod_{n \in \Z} \beta_n^{j_n}\right)\xi \right)\psi =\xi,$$
hence $\p\psi \in \stab_{\L_k}(\xi)$ and
\beq
\label{aves6}
{\rm Aut}_+(\L_k) = \Psi_k\stab_{\L_k}(\xi).
\eeq
Since $\lambda \in \Psi_k$, it follows from Lemma \ref{stdis}(ii) that $\Psi_kA_k \cap  \stab_{\L_k}(\xi)$ is a finite index subgroup of $\stab_{\L_k}(\xi)$. Therefore $\Psi_kA_k$ is a finite index subgroup of ${\rm Aut}_+(\L_k)$ and (iii) holds.

Suppose now that $p_i$ divides $(\frac{k}{p_i})^2 -1$ for $i = 1,\ldots,s$. Then it is straightforward to check that 
\beq
\label{aves5}
\lambda = \rho_1\ldots\rho_s.
\eeq
Indeed, in view of (\ref{aves}) we get
$$\alpha\rho_1 \ldots \rho_s = \alpha^{k}
\beta_{1}^{\sum_{i=1}^s k_i^2} = \beta_{1}^{\sum_{i=1}^s k_i^2},$$
hence it suffices to show that $k$ divides $\sum_{i=1}^s k_i^2 -1$, which is equivalent to $p_i$ dividing $(\frac{k}{p_i})^2 -1$ for $i = 1,\ldots,s$. 

Now it follows from (\ref{aves5}) that $\rho_s \in \Psi_kA_k$, yielding $\stab_{\L_k}(\xi) \leq \Psi_kA_k$. By (\ref{aves6}), we get ${\rm Aut}_+(\L_k) = \Psi_kA_k$ and so ${\rm Aut}_+(\L_k) \cong \L_k \rtimes (\Z^{s-1} \times \Z_k^*)$ by (\ref{aves3}). Therefore ${\rm Aut}(\L_k) \cong (\L_k \rtimes (\Z^{s-1} \times \Z_k^*)) \rtimes C_2$ by (\ref{sd}).

(iii) $\Rw$ (ii). The groups $\L_k$ and $\Z^{s-1} \times \Z_k^*$ are finitely generated, so must be their semidirect product and all their finite extensions. Thus ${\rm Aut}_+(\L_k)$ is finitely generated and so is ${\rm Aut}(\L_k)$ by (\ref{sd}).

(ii) $\Rw$ (i). Suppose now that $k$ is not squarefree. By Lemmas \ref{stanfg} and \ref{coprime}, $\stab_{\L_{k}}(\xi)$ is not finitely generated. By Lemma \ref{autst}, ${\rm Aut}_+(\L_k)$ is not finitely generated either. Since ${\rm Aut}_+(\L_k)$ is a finite index subgroup of ${\rm Aut}(\L_k)$, and every finite index subgroup of a finitely generated subgroup is finitely generated, it follows that ${\rm Aut}_+(\L_k)$ is not finitely generated.
\end{proof}

\section{Uniform continuity}

\begin{theorem}
\label{luc}
Every $\p \in {\rm Aut}(\L_k)$ is uniformly continuous
with respect to the depth metric.
\end{theorem}

\begin{proof}
By Lemma \ref{gene} it is sufficient to show that every
$\theta \in \stab_{\L_k}(\xi)$ and that the maps $\iota$ and $\zeta$ are uniformly continuous 
with respect to the depth metric.

Assume first that $\theta \in \stab_{\L_k}(\xi)$. By
Proposition \ref{ucsta}, it will suffice to show that
\beq
\label{luc2}
(\stab_{p}(\L_k))\theta \subseteq \stab_{p}(\L_k)
\eeq
holds for every $p \geq 0$. 

Write $\alpha\theta = \beta_{i_1}^{r_1}\ldots \beta_{i_m}^{r_m}$. Then
$\beta_{\ell}\theta = \beta_{\ell + i_1}^{r_1}\ldots \beta_{\ell + i_m}^{r_m}$ for every $\ell \in \Z$.
Assume that $\gamma\xi^n \in \stab_{p}(\L_k)$ for some $\gamma \in \fin(\L_k)$ and $n \in \Z$. 
Write $\gamma = \beta_{i'_1}^{r'_1}\ldots \beta_{i'_s}^{r'_s}$. 
Let $q = p+1-\min\{ i'_1,\ldots,i'_s\}$. Since $\stab_{p}(\L_k) \unlhd \L_k$ and $\theta \in \aut(\L_k)$, we have 
\beq
\label{zur2}
\gamma\xi^n \in \stab_{p}(\L_k) \iff (\xi^q\gamma\xi^{-q})\xi^n \in \stab_{p}(\L_k),
\eeq
\beq
\label{zur3}
(\gamma\xi^n)\theta \in \stab_{p}(\L_k) \iff ((\xi^q\gamma\xi^{-q})\xi^n)\theta \in \stab_{p}(\L_k).
\eeq
Therefore we may assume that $i'_v \geq p+1$ for every $v$, replacing $\gamma\xi^n$ by $(\xi^q\gamma\xi^{-q})\xi^n$ if needed.

Write 
$$f(t) = \sum_{v=1}^s r'_v(1-t)^{i'_v}, \quad g(t) = \sum_{u=1}^m r_u(1-t)^{i_u}, \quad
h(t) = (1-t)^{-n}.$$
Considering the formal series $X = \sum_{j=0}^{\infty} x_jt^j$, we have
\beq
\label{eut2}
X\gamma\xi^n = (X + f(t))h(t), \quad X((\gamma\xi^n)\theta) = (X + f(t)g(t))h(t).
\eeq
Indeed, the first equality follows from the formal series interpretation of the $\beta_i$ and $\xi^n$. Since $\theta \in \stab_{\L_k}(\xi)$, it suffices to show that $X(\gamma\theta) = X + f(t)g(t)$. This follows easily from the equality
$$\gamma\theta = (\beta_{i'_1}^{r'_1}\ldots \beta_{i'_s}^{r'_s})\theta = \prod_{u=1}^m \prod_{v = 1}^s \beta_{i'_v+i_u}^{r'_vr_u}.$$

Now let $W_{p+1}$ denote the principal ideal of $\Z_k[[t]]$ generated by $t^{p+1}$. Since $\gamma\xi^n \in \stab_{p}(\L_k)$, then $X\gamma\xi^n=x_0+x_1t^1+\ldots+x_pt^p + \ldots$ and so 
the first equation in (\ref{eut2}) implies that $(X + f(t))h(t) -X = X\gamma\xi^n-X \in W_{p+1}$ for every $X \in \Z_k[[t]]$. Considering the particular case $X = 0$, we get $f(t)h(t) \in W_{p+1}$. 

Using the second equation in (\ref{eut2}), we want to show that
$X(\gamma\xi^n)\theta - X=(X + f(t)g(t))h(t) - X \in W_{p+1}$. Since
$$
(X + f(t)g(t))h(t) - X = (X + f(t))h(t)-X + f(t)(g(t) -1)h(t),$$ 
the claim follows from $(X + f(t))h(t) -X \in W_{p+1}$ and $f(t)h(t) \in W_{p+1}$. 
Since $X(\gamma\xi^n)\theta - X \in W_{p+1}$, then $(\gamma\xi^n)\theta \in \stab_{p}(\L_k)$ and so (\ref{luc2}) holds. Therefore $\theta$ is uniformly continuous for every $\theta \in \stab_{\L_k}(\xi)$.

We show next that $\iota$ is uniformly continuous.
By
Proposition \ref{ucsta}, it suffices to show that
\beq
\label{luc3}
(\stab_{p+1}(\L_k))\iota \subseteq \stab_{p}(\L_k)
\eeq
holds for every $p \geq 0$. 

Assume that $\gamma\xi^n \in \stab_{p+1}(\L_k)$ for some $\gamma \in \fin(\L_k)$ and $n \in \Z$. 
Since $\stab_{p}(\L_k) \leq \L_k$ and $\iota \in \aut(\L_k)$, we have 
$$\gamma\xi^n \in \stab_{p}(\L_k) \iff \xi^{-n}\gamma\inv \in \stab_{p}(\L_k),\quad (\gamma\xi^n)\iota \in \stab_{p}(\L_k) \iff (\xi^{-n}\gamma\inv)\iota \in \stab_{p}(\L_k).$$
Therefore we may assume that $n \geq 0$, replacing $\gamma\xi^n$ by $\xi^{-n}\gamma\inv$ if needed.

Write $\gamma = \beta_{i_1}^{r_1}\ldots \beta_{i_s}^{r_s}$. Since $\iota \in \Psi_k$, then
 (\ref{tail}) implies that $\beta_{\ell}\iota = \beta_{\ell}$ for every $\ell \in \Z$. Thus $\gamma\iota = \gamma$. On the other hand, a straightforward induction shows that 
\beq
\label{eut4}
\xi^n\iota = \beta_0\ldots \beta_{n-1}\xi^n
\eeq
for every $n \geq 0$. Write 
$$f(t) = \sum_{u=1}^m r_u(1-t)^{i_u}, \quad
h(t) = (1-t)^{-n}.$$
Considering the formal series $X = \sum_{j=0}^{\infty} x_jt^j$, we have
\beq
\label{eut3}
X\gamma\xi^n = (X + f(t))h(t), \quad X((\gamma\xi^n)\iota) = (X + f(t) + \sum_{i=0}^{n-1} (1-t)^i)h(t).
\eeq
Indeed, the first equality follows from the formal series interpretation of the $\beta_i$ and $\xi^n$, and the second equality follows from $\gamma\iota = \gamma$ and (\ref{eut4}).

Since $\gamma\xi^n \in \stab_{p+1}(\L_k)$, we have $(X + f(t))h(t) -X \in W_{p+2}$ for every $X \in \Z_k[[t]]$ by the first equality in (\ref{eut3}). Hence $X(h(t)-1) +f(t)h(t) \in W_{p+2}$ for every $X \in \Z_k[[t]]$.
Considering the particular case $X = 0$, we get $f(t)h(t) \in W_{p+2}$ and consequently $X(h(t)-1) \in W_{p+2}$ for every $X \in \Z_k[[t]]$. Since $X$ is arbitrary, it follows that $h(t)-1 \in W_{p+2}$. But then
$$\sum_{i=0}^{n-1} (1-t)^ih(t) = \frac{1-(1-t)^n}{1-(1-t)}h(t) = \frac{h(t)-1}{t} \in W_{p+1},$$
yielding 
$$X((\gamma\xi^n)\theta) -X = (X + f(t) + \sum_{i=0}^{n-1} (1-t)^i)h(t) - X = ((X + f(t))h(t) - X) + \sum_{i=0}^{n-1} (1-t)^ih(t) \in W_{p+1}.$$
Hence $(\gamma\xi^n)\theta \in \stab_{p}(\L_k)$ and so (\ref{luc3}) holds. Therefore $\iota$ is uniformly continuous.

Finally, we show that $\zeta$ is uniformly continuous. By
Proposition \ref{ucsta}, it will suffice to show that
\beq
\label{zur1}
(\stab_{p}(\L_k))\theta \subseteq \stab_{p}(\L_k)
\eeq
holds for every $p \geq 0$. 

Recall that $\alpha\zeta = 
\alpha$ and $\xi \zeta = \xi^{-1}$ and so $\beta_{\ell} \zeta=\beta_{-\ell}$ for every $\ell \in \mathbb{Z}$. Assume now that $\gamma\xi^n \in \stab_{p}(\L_k)$ for some $\gamma = \beta_{i_1}^{r_1}\ldots \beta_{i_m}^{r_m}\in \fin(\L_k)$ and $n \in \Z$. Let $q = \max\{ |i_1|,\ldots,|i_m|\} + 1$. Then (\ref{zur2}) and (\ref{zur3}) (with $\theta$ replaced by $\zeta$) hold, 
so we can assume that $i_1,\ldots,i_m > 0$ (replacing $\gamma\xi^n$ by $(\xi^q\gamma\xi^{-q})\xi^n$ if needed).

Write 
$$f(t) = \sum_{u=1}^m r_u(1-t)^{i_u}, \quad g(t) = \sum_{u=1}^m r_u(1-t)^{-i_u}, \quad
h(t) = (1-t)^{-n}.$$

Recall that $h^{-1}(t)=(1-t)^n$. Considering the formal series $X = \sum_{j=0}^{\infty} x_jt^j$, we have

\beq
\label{zur4}
X\gamma\xi^n = (X + f(t))h(t), \quad X((\gamma\xi^n)\zeta) = (X + g(t))h\inv (t).
\eeq

Indeed, the first equality follows from the formal series interpretation of the $\beta_i$ and $\xi^n$,
and the second equality follows from $\gamma \zeta = \beta_{-i_1}^{r_1}\ldots \beta_{-i_m}^{r_m}$
and $\xi \zeta = \xi^{-1}$ and the same argument used to get the first equality.

Since $\gamma\xi^n \in \stab_{p}(\L_k)$, we have $(X + f(t))h(t) -X \in W_{p+1}$ for every $X \in \Z_k[[t]]$ by (\ref{zur4}). 
Considering the particular case $X = 0$, we get $f(t)h(t) \in W_{p+1}$ and consequently $f(t) \in W_{p+1}$ (since $h(t) = 1 + th'(t)$ for some $h'(t) \in \Z_k[[t]]$). We shall prove that $g(t) \in W_{p+1}$.

The coefficient of $t^{\ell}$ in $f(t)$ is $(-1)^{\ell} \sum_{u=1}^m r_u\binom{i_u}{\ell}$. Since $f(t) \in W_{p+1}$, we get
\beq
\label{zur5}
\sum_{u=1}^m r_u\binom{i_u}{\ell} = 0 \; \mbox{ for }\ell = 0,\ldots,p.
\eeq
For $\ell = 0$, this is equivalent to $\sum_{u=1}^m r_u = 0$. 
Since $i_u \ge 1$, for every $s \geq -1$ we have $\binom{i_u+s}{0}=1$ and so
\beq
\label{zur6}
\sum_{u=1}^m r_u\binom{i_u+s}{0} = 0.
\eeq
We show next that
\beq
\label{zur7}
\sum_{u=1}^m r_u\binom{i_u+s}{\ell} = 0
\eeq
for all $\ell \in \{ 1,\ldots,p\}$ and $s \geq 0$. We use induction on $s$.

The case $s = 0$ follows from (\ref{zur5}), hence we assume that $s > 0$ and (\ref{zur7}) holds for $s-1$ and every $\ell \in \{ 1,\ldots,p\}$. Indeed, we may use the induction hypothesis and (\ref{zur6}) (if $\ell = 1$) to deduce
$$\sum_{u=1}^m r_u\binom{i_u+s}{\ell} = \sum_{u=1}^m r_u\binom{i_u+s-1}{\ell} + \sum_{u=1}^m r_u\binom{i_u+s-1}{\ell-1} = 0 + 0 = 0.$$
Thus (\ref{zur7}) holds for every $s$. 

We compute now the coefficient $c_{\ell}$ of $t^{\ell}$ in $(1-t)^{-i_u} = (1+t+t^2+\ldots)^{i_u}$. Then $c_{\ell}$ is the number of decompositions of the form $\ell = a_1 + \ldots + a_{i_u}$ with $a_1,\ldots,a_{i_u} \in \N$, and therefore also the number of decompositions of the form $\ell + i_u = b_1 + \ldots + b_{i_u}$ with $b_1,\ldots,b_{i_u} \geq 1$. This number is well known to be $\binom{i_u +\ell -1}{i_u-1} =  \binom{i_u+\ell -1}{\ell}$ (we must choose $i_u+\ell -1$ intermediate positions in a sequence of $\ell +i_u$ 1's to bound the summands).
It follows that the coefficient of $t^{\ell}$ in $g(t)$ is $\sum_{u=1}^m r_u\binom{i_u + \ell -1}{\ell}$. By (\ref{zur6}) and (\ref{zur7}), we get $\sum_{u=1}^m r_u\binom{i_u + \ell -1}{\ell} = 0$ for $\ell = 0,\ldots,p$. Thus $g(t) \in W_{p+1}$ as claimed.

Now 
$$\begin{array}{lll}
(X + g(t))h\inv (t) - X&=&g(t)h\inv (t) -(X-Xh\inv (t)) = g(t)h\inv (t) -(Xh(t) -X)h\inv (t)\\
&=&g(t)h\inv (t) -((X+f(t))h(t) -X)h\inv (t) + f(t) \in W_{p+1}
\end{array}$$
since $W_{p+1}$ is an ideal of $\Z_k[[t]]$. By (\ref{zur4}), we get $(\gamma\xi^n)\zeta \in \stab_{p}(\L_k)$ and so (\ref{zur1}) holds. Therefore $\zeta$ is uniformly continuous.

Since $\iota^k = \iota$, $\zeta^2 = \zeta$ and the composition of uniformly continuous mappings is uniformly continuous, it follows from Lemma \ref{gene}(ii) that every $\p \in \aut(\L_k)$ is uniformly continuous.
\end{proof}

Next we note that the fixed point subgroup of an automorphism of $\L_2$
needs not be finitely generated, even if the automorphism is inner.

\begin{example}
\label{nff}
There exists $\theta \in {\rm Inn}(\L_2)$ such that ${\rm Fix}(\theta)$
is not finitely generated.
\end{example}

\noindent Indeed, let $\theta \in \inn(\L_2)$ be defined by
$$\psi\theta = \alpha\psi\alpha\inv \quad (\psi \in \L_2).$$
Since $\fin(\L_2)$ is abelian, we have
$$\gamma\theta = \alpha\gamma\alpha = \alpha^2\gamma = \gamma$$
for every $\gamma \in \fin(\L_2)$. On the other hand, given $n \in \Z$,
we have
$$\xi^n\theta = \alpha\xi^n\alpha = \beta_0\beta_n\xi^n,$$
hence $\xi^n \in \fix(\theta)$ if and only if $n = 0$.  It follows that $\fix(\theta) = \fin(\L_2)$, an infinite abelian torsion
group, therefore non finitely generated.

\begin{questions}
We conclude this paper with some open questions and suggestions for further research.
Let $G \leq \aut(T_A)$ be an automaton group and $\overline{G}$ its closure
with respect to the depth metric. 

\begin{enumerate}
\item
Does every automorphism of of an automaton group $G \leq \aut(T_A)$
extend to an automorphism of $\aut(T_A)$? 
\item Is every fixed point of $\overline{\varphi}$ a limit point of fixed points of $\varphi$?
\end{enumerate}
\end{questions}


\section*{Acknowledgements}

The authors would like to thank Manuel Delgado for providing support in some calculations with the
\texttt{GAP} mathematical software and an anonymous referee for helpful
comments and references which enriched the exposition.

The first author is a member of the Gruppo Nazionale per le Strutture Algebriche, Geometriche e le loro Applicazioni (GNSAGA) of the Istituto Nazionale di Alta Matematica (INdAM) and gratefully acknowledges the support of the 
Funda\c{c}\~ao de Amparo \`a Pesquisa do Estado de S\~ao Paulo 
(FAPESP Jovens Pesquisadores em Centros Emergentes grant 2016/12196-5),
of the Conselho Nacional de Desenvolvimento Cient\'ifico e Tecnol\'ogico (CNPq 
Bolsa de Produtividade em Pesquisa PQ-2 grant 306614/2016-2) and of
the Funda\c{c}\~ao para a Ci\^encia e a Tecnologia  (CEMAT-Ci\^encias FCT project UID/Multi/04621/2013).

The second author was partially supported by CMUP
(UID/MAT/00144/2019), which is funded by FCT (Portugal) with national
(MCTES) and European structural funds through the programs FEDER, under
the partnership agreement PT2020. 

\bibliographystyle{plain}
\bibliography{bibExtAutSelfSimilar}

\bigskip

\sc{Francesco Matucci, 
Dipartimento di Matematica e Applicazioni, Universit\`a degli Studi di Milano-Bicocca, Office 3043,  Via Cozzi 55, 20125 Milano, Italy
}

{\em E-mail address:} \texttt{francesco.matucci@unimib.it}

%

\bigskip

{\sc Pedro V. Silva, Centro de
Matem\'{a}tica, Faculdade de Ci\^{e}ncias, Universidade do
Porto, R. Campo Alegre 687, 4169-007 Porto, Portugal}

{\em E-mail address}: \texttt{pvsilva@fc.up.pt}

\end{document}